\documentclass{amsproc}
\usepackage[arrow, curve, frame, matrix, graph, tips]{xy}
\usepackage{amssymb}
\usepackage{amsthm}
\usepackage{amsmath}
\usepackage{enumerate}

\newtheorem{theorem}{Theorem}
\newtheorem{lemma}[theorem]{Lemma}
\newtheorem{proposition}[theorem]{Proposition}
\newtheorem{corollary}[theorem]{Corollary}

\theoremstyle{definition}
\newtheorem{definition}[theorem]{Definition}
\newtheorem{example}[theorem]{Example}

\newtheorem{non-example}[theorem]{Non-Example}

\newtheorem{question}[theorem]{Question}

\theoremstyle{remark}
\newtheorem{remark}[theorem]{Remark}

\newcommand{\tn}[1]{\textnormal{#1}}
\newcommand{\tnb}[1]{\textnormal{\bf #1}}

\newcommand{\C}{\mathbb{C}}
\newcommand{\D}{\mathbb{D}}

\newcommand{\N}{\mathbb{N}}

\newcommand{\id}{\tn{id}}
\newcommand{\ca}[1]{\mathcal{#1}}
\newcommand{\ladj}{\dashv}
\newcommand{\iso}{\cong}

\newcommand{\Set}{\tnb{Set}}

\newcommand{\CAT}{\tnb{CAT}}
\newcommand{\op}{\tn{op}}

\renewcommand{\iff}{\Leftrightarrow}

\DeclareMathOperator*{\Tbar}{\overline{T}}
\DeclareMathOperator*{\Sbar}{\overline{S}}

\DeclareMathOperator*{\opE}{E}
\DeclareMathOperator*{\opEpr}{E^{\prime}}
\DeclareMathOperator*{\opF}{F}

\DeclareMathOperator*{\opEr}{E^{(r)}}

\DeclareMathOperator*{\opEm}{E^{(m)}}
\DeclareMathOperator*{\opEmpone}{E^{(m{+}1)}}
\DeclareMathOperator*{\opEmptwo}{E^{(m{+}2)}}
\DeclareMathOperator*{\opErpone}{E^{(r{+}1)}}

\DeclareMathOperator*{\opEzero}{E^{(0)}}
\DeclareMathOperator*{\opEone}{E^{(1)}}
\DeclareMathOperator*{\colsum}{\colim_{r{<}m}\coprod}

\newcommand{\Enrich}[1]{#1{\textnormal{-Cat}}}

\newcommand{\col}{\tn{col}}

\DeclareMathOperator*{\colim}{\textnormal{colim}}

\newcommand{\Alg}[1]{#1{\textnormal{-Alg}}}

\newcommand{\MND}{\textnormal{MND}}
\newcommand{\OpMND}{\textnormal{OpMND}}

\newcommand{\DISTMULT}{\textnormal{DISTMULT}}
\newcommand{\OpDISTMULT}{\textnormal{OpDISTMULT}}

\newcommand{\con}{\tn{con}}

\begin{document}

\title{The Lifting Theorem for Multitensors}

\author{Michael Batanin}
\address{Department of Mathematics,
Macquarie University}
\email{mbatanin@ics.mq.edu.au}
\thanks{}
\author{Denis-Charles Cisinski}
\address{Departement des Mathematiques,
Universit\'e Paris 13 Villanteuse}
\email{cisinski@math.paris13.fr}
\thanks{}
\author{Mark Weber}
\address{Department of Mathematics,
Macquarie University}
\email{mark.weber.math@gmail.com}
\thanks{}
\maketitle
\begin{abstract}
We continue to develop the theory of \cite{EnHopI} and \cite{WebMultMnd} on monads and multitensors. The central result of this paper -- the lifting theorem for multitensors -- enables us to see the Gray tensor product of 2-categories and the Crans tensor product of Gray categories as part of our emerging framework. Moreover we explain how our lifting theorem gives an alternative description of Day convolution \cite{DayConvolution} in the unenriched context.
\end{abstract}
\tableofcontents

\section{Multitensors and functor operads}

This paper continues the developments of \cite{EnHopI} and \cite{WebMultMnd} on the interplay between monads and multitensors in the globular approach to higher category theory. To take an important example, according to \cite{WebMultMnd} there are two related combinatorial objects which can be used to describe the notion of Gray category. One has the monad $A$ on the category $\ca G^3(\Set)$ of 3-globular sets whose algebras are Gray categories, which was first described in \cite{Bat98}. On the other hand there is a multitensor (ie a lax monoidal structure) on the category $\ca G^2(\Set)$ of $2$-globular sets, such that categories enriched in $E$ are exactly Gray categories. The theory described in \cite{WebMultMnd} explains how $A$ and $E$ are related as part of a general theory which applies to all operads of the sort defined originally in \cite{Bat98}.

However there is a third object which is missing from this picture, namely, the Gray tensor product of 2-categories. It is a simpler object than $A$ and $E$, and categories enriched in $\Enrich 2$ for the Gray tensor product are exactly Gray categories. The purpose of this paper is to exhibit the Gray tensor product as part of our emerging framework. This is done by means of the lifting theorem for multitensors -- theorem(\ref{thm:lift-mult}) of this article.

Recall \cite{EnHopI,WebMultMnd} that a \emph{multitensor} $(E,u,\sigma)$ on a category $V$ consists of n-ary tensor product functors $E_n:V^n \to V$, whose values on objects are denoted in any of the following ways
\[ \begin{array}{ccccccc} {E(X_1,...,X_n)} && {E_n(X_1,...,X_n)} && {\opE\limits_{1{\leq}i{\leq}n} X_i} && {\opE\limits_i X_i} \end{array} \]
depending on what is most convenient, together with unit and substitution maps
\[ \begin{array}{lcr} {u_X:Z \rightarrow E_1X} && {\sigma_{X_{ij}}:\opE\limits_i\opE\limits_j X_{ij} \rightarrow \opE\limits_{ij} X_{ij}} \end{array} \]
for all $X$, $X_{ij}$ from $V$ which are natural in their arguments and satisfy the obvious unit and associativity axioms. It is also useful to think of $(E,u,\sigma)$ more abstractly as a lax algebra structure on $V$ for the monoid monad $M$ on $\CAT$, and so to denote $E$ as a functor $E:MV \to V$. The basic example to keep in mind is that of a monoidal structure on $V$, for in this case $E$ is given by the $n$-ary tensor products, $u$ is the identity and the components of $\sigma$ are given by coherence isomorphisms for the monoidal structure.

A \emph{category enriched in $E$} consists of a $V$-enriched graph $X$ together with composition maps
\[ \kappa_{x_i} : \opE\limits_i X(x_{i-1},x_i) \rightarrow X(x_0,x_n) \]
for all $n \in \N$ and sequences $(x_0,...,x_n)$ of objects of $X$, satisfying the evident unit and associativity axioms. With the evident notion of $E$-functor (see \cite{EnHopI}), one has a category $\Enrich E$ of $E$-categories and $E$-functors together with a forgetful functor
\[ U^E : \Enrich E \rightarrow \ca GV. \]
When $E$ is a distributive multitensor, that is when $E_n$ commutes with coproducts in each variable, one can construct a monad $\Gamma E$ on $\ca GV$ over $\Set$. The object map of the underlying endofunctor is given by the formula
\[ \Gamma EX(a,b) = \coprod\limits_{a=x_0,...,x_n=b} \opE\limits_iX(x_{i-1},x_i), \]
the unit $u$ is used to provide the unit of the monad and $\sigma$ is used to provide the multiplication. The identification of the algebras of $\Gamma E$ and categories enriched in $E$ is witnessed by a canonical isomorphism $\Enrich E \iso \ca G(V)^{\Gamma E}$ over $\ca GV$. This construction, the senses in which it is 2-functorial, and its respect of various categorical properties, is explained fully in \cite{WebMultMnd}. We use the notation and terminology of \cite{WebMultMnd} freely.

If one restricts attention to unary operations, then $E_1$, $u$ and the components $\sigma_X:E_1^2X \to E_1X$ provide the underlying endofunctor, unit, and multiplication for a monad on $V$. This monad is called the \emph{unary part} of $E$. When the unary part of $E$ is the identity monad, the multitensor is a \emph{functor operad}. This coincides with existing terminology, see \cite{McClureSmith} for instance, except that we don't in this paper consider any symmetric group actions. Since units for functor operads are identities, we denote any such as a pair $(E,\sigma)$, where as for general multitensors $E$ denotes the functor part and $\sigma$ the substitution.

By definition then, a functor operad is a multitensor. On the other hand, as observed in \cite{EnHopI} lemma(2.7), the unary part of a multitensor $E$ acts on $E$, in the sense that as a functor $E$ factors as
\[ \xygraph{{MV}="l" [r] {V^{E_1}}="m" [r] {V}="r" "l":"m"^-{}:"r"^-{U^{E_1}}} \]
and in addition, the substitution maps are morphisms of $E_1$-algebras. Moreover an $E$-category structure on a $V$-enriched graph $X$ includes in particular an $E_1$-algebra structure on each hom $X(a,b)$ of $X$ with respect to which the composition maps are morphisms of $E_1$-algebras. These observations lead to
\begin{question}\label{q:lift}
Given a multitensor $(E,u,\sigma)$ on a category $V$ can one find a functor operad $(E',\sigma')$ on $V^{E_1}$ such that $E'$-categories are exactly $E$-categories?
\end{question}
The main result of this paper, theorem(\ref{thm:lift-mult}), says that question(\ref{q:lift}) has a nice answer: when $E$ is distributive and accessible and $V$ is cocomplete, one can indeed find a \emph{unique} distributive accessible such $E'$. Moreover as we will see in section(\ref{sec:convolution}), this construction generalises Day convolution \cite{DayConvolution} and some of its lax analogues \cite{DS03}.

Perhaps the first appearance of a case of our lifting theorem in the literature, that does not involve convolution, is in the work of Ginzburg and Kapranov on Koszul duality \cite{GinKap}. Formula (1.2.13) of that paper, in the case of a $K$-collection $E$ coming from an operad, implicitly involves the lifting of the multitensor corresponding (as in \cite{EnHopI} example(2.6)) to the given operad. For instance, our lifting theorem gives a satisfying general explanation for why one must tensor over $K$ in that formula.

This paper is organised in the following way. The lifting theorem is proved in section(\ref{sec:lifting-theorem}), using some transfinite constructions from monad theory which are recalled in appendix(\ref{sec:transfinite}). The lifted functor operad is unpacked explicitly in section(\ref{sec:explicit-lifting}). In section(\ref{sec:Gray-Crans}) we explain how the Gray tensor product of 2-categories and Crans tensor product of Gray categories is obtained as a lifting via theorem(\ref{thm:lift-mult}). Part of the interplay between monads and multitensors described in \cite{WebMultMnd} covers contractible multitensors and their relation to the contractible operads of \cite{Bat98}. In section(\ref{sec:contractibility}) we extend this analysis to the lifted multitensors, and in example(\ref{ex:Gray-contractible}) explain how this gives a different proof of the contractibility of the operad for Gray categories. In section(\ref{sec:convolution}) we explain how Day convolution, in the unenriched setting, can also be obtained via our theorem(\ref{thm:lift-mult}).

\section{The lifting theorem}\label{sec:lifting-theorem}
The idea which enables us to answer question(\ref{q:lift}) is the following. Given a distributive multitensor $E$ on $V$ one can consider also the multitensor $\widetilde{E_1}$ whose unary part is also $E_1$, but whose non-unary parts are all constant at $\emptyset$. This is clearly a sub-multitensor of $E$, also distributive, and moreover as we shall see one has $\Enrich{\widetilde{E_1}} \iso \ca G(V^{E_1})$ over $\ca GV$. Thus from the inclusion $\widetilde{E_1} \hookrightarrow E$ one induces the forgetful functor $U$ fitting in the commutative triangle
\[ \xygraph{{\ca G(V^{E_1})}="l" [r(2)] {\Enrich E}="r" [dl] {\ca GV}="b" "l":@{<-}"r"^-{U}:"b"^-{U^E}:@{<-}"l"^-{\ca G(U^{E_1})}} \]
For sufficiently nice $V$ and $E$ this forgetful functor has a left adjoint. The category of algebras of the induced monad $T$ will be $\Enrich E$ since $U$ is monadic. Thus problem is reduced to that of establishing that this monad $T$ arises from a multitensor on $V^{E_1}$. Theorem(42) of \cite{WebMultMnd} gives the properties that $T$ must satisfy in order that there is such a multitensor, and gives an explicit formula for it in terms of $T$.

In the interplay between multitensors and monads described in \cite{WebMultMnd} the construction $E \mapsto \Gamma E$ of a monad on $\ca GV$ over $\Set$ from a distributive multitensor provides the object map of 2-functors
\[ \begin{array}{l} {\Gamma : \DISTMULT \to \MND(\CAT/\Set)} \\ {\Gamma' : \OpDISTMULT \to \OpMND(\CAT/\Set).} \end{array} \]
That the monads $(S,\eta,\mu)$ on $\ca GV$ that arise from this construction are ``over $\Set$'' means that for all $X \in \ca GV$, the $V$-graph $SX$ has the same object set as $X$, and the components of the unit $\eta$ and multiplication $\mu$ are identities on objects. Theorem(42) of \cite{WebMultMnd} alluded to above characterises the monads on $\ca GV$ over $\Set$ of the form $\Gamma E$ as those which are \emph{distributive} and \emph{path-like} in the sense of definitions(41) and (38) of \cite{WebMultMnd} respectively. Note that the properties of distributivity and path-likeness concern only the functor part of a given monad on $\ca GV$ over $\Set$. On the way to the proof of theorem(\ref{thm:lift-mult}) below, it is necessary to have available these definitions for functors over $\Set$ between categories of enriched graphs.

Suppose that categories $V$ and $W$ have coproducts. Recall that a finite sequence $(Z_1,...,Z_n)$ of objects of $V$ may be regarded as a $V$-graph whose object set is $\{0,...,n\}$, hom from $(i-1)$ to $i$ is $Z_i$ for $1 \leq i \leq n$, and other homs are initial. Then a functor $T:\ca GV \to \ca GW$ over $\Set$ determines a functor $\overline{T}:MV \to W$ whose object map is given by
\[ \overline{T}(Z_1,...,Z_n) = T(Z_1,...,Z_n)(0,n). \]
By definition $\overline{T}$ amounts to functors $\overline{T}_n:V^n \to W$ for each $n \in \N$, and one may consider the various categorical properties that such $\overline{T}$ may enjoy, as in the discussion of \cite{WebMultMnd} section(4.3).
\begin{definition}\label{def:distributive}
Let $V$ and $W$ be categories with coproducts. A functor $T:\ca GV \to \ca GW$ over $\Set$ is \emph{distributive} when for each $n \in \N$, $\overline{T}_n$ preserves coproducts in each variable.
\end{definition}
\noindent Given a $V$-graph $X$ and sequence $x=(x_0,...,x_n)$ of objects of $X$, one can define the morphism
\[ \overline{x} : (X(x_0,x_1),X(x_1,x_2),...,X(x_{n-1},x_n)) \to X \]
whose object map is $i \mapsto x_i$, and whose hom map between $(i-1)$ and $i$ is the identity. For all such sequences $x$ one has
\[ T(\overline{x})_{0,n} : \Tbar\limits_i X(x_{i-1},x_i) \to TX(x_0,x_n) \]
and so taking all sequences $x$ starting at $a$ and finishing at $b$ one induces the canonical map
\[ \pi_{T,X,a,b} : \coprod\limits_{a=x_0,...,x_n=b} \Tbar\limits_i X(x_{i-1},x_i) \rightarrow TX(a,b) \]
in $W$.
\begin{definition}\label{def:path-like}
Let $V$ and $W$ be categories with coproducts. A functor $T:\ca GV \to \ca GW$ over $\Set$ is \emph{path-like} when for all $X \in \ca GV$ and $a,b \in X_0$, the maps $\pi_{T,X,a,b}$ are isomorphisms.
\end{definition}
\noindent Clearly a monad $(T,\eta,\mu)$ on $\ca GV$ over $\Set$ is distributive (resp. path-like) in the sense of \cite{WebMultMnd} iff the underlying endofunctor $T$ is so in the sense just defined.
\begin{lemma}\label{lem:dpl}
Let $V$, $W$ and $Y$ be categories with coproducts and $R:V \to W$, $T:\ca GV \to \ca GW$ and $S:\ca GW \to \ca GY$ be functors.
\begin{enumerate}
\item If $R$ preserves coproducts then $\ca GR$ is distributive and path-like.\label{lemcase:dpl-unary}
\item If $S$ and $T$ are distributive and path-like, then so is $ST$.\label{lemcase:dpl-composition}
\end{enumerate}
\end{lemma}
\begin{proof}
(\ref{lemcase:dpl-unary}): Since $R$ preserves the initial object one has $\ca GR(Z_1,...,Z_n) = (RZ_1,...,RZ_n)$ and so $\overline{\ca GR}:MV \to W$ sends sequences of length $n \neq 1$ to $\emptyset$, and its unary part is just $R$. Thus $\ca GR$ is distributive since $R$ preserves coproducts, and coproducts of copies of $\emptyset$ are initial. The summands of the domain of $\pi_{\ca GR,X,a,b}$ are initial unless $(x_0,...,x_n)$ is the sequence $(a,b)$, thus $\pi_{\ca GR,X,a,b}$ is clearly an isomorphism, and so $\ca GR$ is path-like.

(\ref{lemcase:dpl-composition}): Since $S$ and $T$ are path-like and distributive one has
\[ ST(Z_1,...,Z_n)(0,n) \iso \coprod\limits_{0=r_0{\leq}...{\leq}r_m=n} \Sbar\limits_{1{\leq}i{\leq}m} \,\,\, \Tbar\limits_{r_{i-1}{<}j{\leq}r_i} Z_j  \]
and so $ST$ is path-like and distributive since $S$ and $T$ are, and since a coproduct of coproducts is a coproduct. 
\end{proof}
In order to implement the above strategy, we must know something about the explicit description of the left adjoint $\ca G(V^{E_1}) \to \Enrich E$. Since this may be regarded as arising from the monad morphism $\phi$ induced by the inclusion $\tilde{E_1} \hookrightarrow E$ via $\Gamma$, the well-known construction of $\phi_!$ reviewed in section(\ref{ssec:Dubuc}) is what we need. The key point about this construction is that it proceeds via a transfinite process involving only \emph{connected} colimits. Moreover in the context that we shall soon consider, these will be connected colimits of diagrams of $V$-graphs which live wholly within a single fibre of $(-)_0:\ca GV \to \Set$. By remark(21) of \cite{WebMultMnd}, such colimits in $\ca GV$ may be computed by taking the same object set, and computing the colimit one hom at a time in the expected way. The importance of this is underscored by
\begin{lemma}\label{lem:concol-pathlike}
Let $V$ be a category with coproducts, $W$ be a cocomplete category, $J$ be a small connected category and
\[ F : J \rightarrow [\ca GV,\ca GW] \]
be a functor. Suppose that $F$ sends objects and arrows of $J$ functors and natural transformations over $\Set$.
\begin{itemize}
\item[(1)]  Then the colimit $K:\ca GV{\rightarrow}\ca GW$ of $F$ may be chosen to be over $\Set$.\label{cpl1}
\end{itemize}
Given such a choice of $K$:
\begin{itemize}
\item[(2)]  If $Fj$ is path-like for all $j \in J$, then $K$ is also path-like.\label{cpl2}
\item[(3)]  If $Fj$ is distributive for all $j \in J$, then $K$ is also distributive.\label{clp3}
\end{itemize}
\end{lemma}
\begin{proof}
Colimits in $[\ca GV,\ca GW]$ are computed componentwise from colimits in $\ca GW$ and so for $X \in \ca GV$ we must describe a universal cocone with components
\[ \kappa_{X,j} : Fj(X) \rightarrow KX. \]
By remark(21) of \cite{WebMultMnd} we may demand that the $\kappa_{X,j}$ are identities on objects, and then compute the hom of the colimit between $a,b \in X_0$ by taking a colimit cocone
\[ \{\kappa_{X,j}\}_{a,b} : Fj(X)(a,b) \rightarrow KX(a,b) \]
in $W$. This establishes (1). Since the properties of path-likeness and distributivity involve only colimits at the level of the homs as does the construction of $K$ just given, (2) and (3) follow immediately since colimits commute with colimits in general.
\end{proof}
Recall the structure-semantics result of Lawvere, which says that for any category $\ca E$, the canonical functor
\[ \begin{array}{lccccc} {\tn{Mnd}(\ca E)^{\op} \to \CAT/\ca E} &&& {T} & {\mapsto} & {U:\ca E^T \to \ca E} \end{array} \]
with object map indicated is fully faithful (see \cite{Str72} for a proof). An important consequence of this is that for monads $S$ and $T$ on $\ca E$, an isomorphism $\ca E^T \iso \ca E^S$ over $\ca E$ is induced by a unique isomorphism $S \iso T$ of monads. We now have all the pieces we need to implement our strategy. First, in the following lemma, we give the result we need to recognise the induced monad on $\ca G(V^{E_1})$ as arising from a multitensor.
\begin{lemma}\label{lem:mnd-lift-mult}
Let $\lambda$ be a regular cardinal. Suppose that $V$ is a cocomplete category, $R$ is a coproduct preserving monad on $V$, $S$ is a $\lambda$-accessible monad on $\ca GV$ over $\Set$, and $\phi:\ca GR{\rightarrow}S$ is a monad morphism over $\Set$. Denote by $T$ the monad on $\ca G(V^R)$ induced by $\phi_! \ladj \phi^*$.
\begin{itemize}
\item[(1)]  One may choose $\phi_!$ so that $T$ is over $\Set$.
\end{itemize}
Given such a choice of $\phi_!$:
\begin{itemize}
\item[(2)]  If $S$ is distributive and path-like then so is $T$.
\item[(3)]  If $R$ is $\lambda$-accessible then so is $T$.
\end{itemize}
\end{lemma}
\begin{proof}
Let us denote by $\rho:RU^R \to U^R$ the 2-cell datum of the Eilenberg-Moore object for $R$, and note that by \cite{WebMultMnd} lemma(16) one may identify $U^{\ca GR}=\ca G(U^R)$ and $\ca G\rho$ as the 2-cell datum for $\ca GR$'s Eilenberg-Moore object. Now $T$ is over $\Set$ iff $\ca G(U^R)T$ is. Moreover since $R$ preserves coproducts $U^R$ creates them, and so $T$ is path-like and distributive iff $\ca G(U^R)T$ is. Since $\ca G(U^R)T = U^S\phi_!$, it follows that $T$ is over $\Set$, path-like and distributive iff $U^S\phi_!$ is. From corollary(\ref{cor:explicit-phi-shreik}) the beginning of the transfinite construction in $[\ca G(V^R),\ca GV]$ giving $U^S\phi_!$ is depicted in
\[ \xygraph{!{0;(1.5,0):(0,.666)::} {S^3\ca G(RU^R)}="tA" [r(2)] {S^3\ca G(U^R)}="tB" [r] {S^2Q_1}="t0" [r] {S^2Q_2}="t1" [r] {...}="t2" 
"tA" [d] {S^2\ca G(RU^R)}="mA" [r(2)] {S^2\ca G(U^R)}="mB" [r] {SQ_1}="m0" [r] {SQ_2}="m1" [r] {...}="m2"
"mA" [d] {S\ca G(RU^R)}="bA" [r(2)] {S\ca G(U^R)}="bB" [r] {Q_1}="b0" [r] {Q_2}="b1" [r] {...}="b2"
"tA":@<-1ex>"tB"_-{S^3\ca G(\rho)} "tA":@<1ex>"tB"^-{S^2\mu^SS(\phi)\ca G(U^R)}:"t0"^-{}:"t1"^-{}:"t2"^-{}
"mA":@<-1ex>"mB"_-{S^2\ca G(\rho)} "mA":@<1ex>"mB"^-{S\mu^SS(\phi)\ca G(U^R)}:"m0"^-{}:"m1"^-{}:"m2"^-{}
"bA":@<-1ex>"bB"_-{S\ca G(\rho)} "bA":@<1ex>"bB"^-{\mu^SS(\phi)\ca G(U^R)}:"b0"_-{}:"b1"_-{}:"b2"_-{}
"tA":@<-1ex>"mA"_-{}:@<-1ex>"bA"_-{} "tA":@<1ex>"mA"^-{Ta}:@<1ex>@{<-}"bA"^-{}
"tB":@<-1ex>"mB"_-{}:@<-1ex>"bB"_-{} "tB":@<1ex>"mB"^-{}:@<1ex>@{<-}"bB"^-{}
"t0":"m0"^-{}:@{<-}"b0"^-{} "t1":"m1"^-{}:@{<-}"b1"^-{}
"tB":"m0"^-{} "mB":"b0"^-{} "t0":"m1"^-{} "m0":"b1"^-{} "t1":"m2"^-{} "m1":"b2"^-{}} \]
Since the monads $S$ and $\ca G(R)$ are over $\Set$, as are $\rho$ and $\phi$, it follows by a transfinite induction using lemma(\ref{lem:concol-pathlike}) that all successive stages of this construction give functors and natural transformations over $\Set$, whence $U^S\phi_!$ is itself over $\Set$. Lemma(\ref{lem:dpl}) ensures that the functors $\ca G(R)$ and $\ca G(RU^R)$ are distributive and path-like, since $R$ preserves coproducts and $U^R$ creates them. When $S$ is also distributive and path-like, then by the same sort of transfinite induction using lemmas(\ref{lem:dpl}) and (\ref{lem:concol-pathlike}), all successive stages of this construction give functors that are distributive and path-like, whence $U^S\phi_!$ is itself distributive and path-like.

Supposing $R$ to be $\lambda$-accessible, note that $\ca G(R)$ is also $\lambda$-accessible. One way to see this is to consider the distributive multitensor $\tilde{R}$ on $V$ whose unary part is $R$ and non-unary parts are constant at $\emptyset$. Thus $\tilde{R}$ will be $\lambda$-accessible since $R$ is. To give an $\tilde{R}$-category structure on $X \in \ca GV$ amounts to giving $R$-algebra structures to the homs of $X$, and similarly on morphisms, whence one has a canonical isomorphism $\Enrich{\tilde{R}} \iso \ca G(V^R)$ over $\ca GV$. By proposition(26) of \cite{WebMultMnd} together with structure-semantics one obtains $\Gamma \tilde{R} \iso \ca GR$. Thus by \cite{WebMultMnd} theorem(29), $\ca GR$ is indeed $\lambda$-accessible. But then it follows that $U^{\ca GR} = \ca G(U^R)$ creates $\lambda$-filtered colimits, and so $T$ is $\lambda$-accessible iff $\ca G(U^R)T = U_S\phi_!$ is. In the transfinite construction of $U^S\phi_!$, it is now clear that the functors involved at every stage are $\lambda$-accessible by yet another transfinite induction, and so $U_S\phi_!$ is $\lambda$-accessible as required.

To finish the proof we must check that $T$'s monad structure is over $\Set$. Since $\mu^T$ is a retraction of $\eta^TT$ it suffices to verify that $\eta^T$ is over $\Set$, which is equivalent to asking that the components of $\ca G(U^R)\eta^T$ are identities on objects. Writing $q:S{\rightarrow}U^S\phi_{!}$ for the transfinite composite constructed as part of the definition of $\phi_!$ recall from the end of section(\ref{ssec:Dubuc}) that one has a commutative square
\[ \xymatrix{{\ca G(RU^R)} \ar[r]^-{\ca G\rho} \ar[d]_{{\phi}\ca G(U^R)} & {\ca GU^R} \ar[d]^{\ca G(U^R)\eta^T} \\ {S\ca G(U^R)} \ar[r]_-{q} & {U^S\phi_{!}}} \]
Now $\ca G\rho$ and $\phi \ca G(U^R)$ are over $\Set$ by definition, and $q$ is by construction, so the result follows.
\end{proof}
\begin{theorem}\label{thm:lift-mult}
\emph{({\bf Multitensor lifting theorem})} Let $\lambda$ be a regular cardinal and let $E$ be a $\lambda$-accessible distributive multitensor on a cocomplete category $V$. Then there is, to within isomorphism, a unique functor operad $(E',\sigma')$ on $V^{E_1}$ such that
\begin{enumerate}
\item  $(E',\sigma')$ is distributive.
\item  $\Enrich {E'} \iso \Enrich E$ over $\ca GV$.
\end{enumerate}
Moreover $E'$ is also $\lambda$-accessible.
\end{theorem}
\begin{proof}
Write $\psi:\widetilde{E_1} \hookrightarrow E$ for the multitensor inclusion of the unary part of $E$, and then apply lemma(\ref{lem:mnd-lift-mult}) with $S=\Gamma{E}$, $R=E_1$ and $\phi=\Gamma{\psi}$ to produce a $\lambda$-accessible distributive and path-like monad $T$ on $\ca G(V^{E_1})$ over $\Set$. Thus by \cite{WebMultMnd} proposition(40) and theorem(42), $\overline{T}$ is a distributive multitensor on $V^{E_1}$ with $\Enrich {\overline{T}} \iso \Enrich E$. Moreover since $T \iso \Gamma {\overline{T}}$ it follows by \cite{WebMultMnd} theorem(29) that $\overline{T}$ is $\lambda$-accessible. As for uniqueness suppose that $(E',\sigma')$ is given as in the statement. Then by \cite{WebMultMnd} theorem(42) $\Gamma(E')$ is a distributive monad on $\ca G(V^{E_1})$ and one has
\[ \ca G(V^{E_1})^{\Gamma(E')} \iso \Enrich E \]
over $\ca G(V^{E_1})$. By structure-semantics one has an isomorphism $\Gamma(E'){\iso}T$ of monads, and thus by \cite{WebMultMnd} proposition(43), an isomorphism $E'{\iso}\overline{T}$ of multitensors.
\end{proof}
%

\section{Multitensor lifting made explicit}\label{sec:explicit-lifting}

Let us now instantiate the constructions of section(\ref{ssec:Dubuc}) to produce a more explicit description of the functor operad $E'$ produced by theorem(\ref{thm:lift-mult}). Beyond mere instantiation this task amounts to reformulating everything in terms of hom maps which live in $V$, because in our case the colimits being formed in $\ca GV$ at each stage of the construction are connected colimits diagrams whose morphisms are all identity on objects. Moreover these fixed object sets are of the form $\{0,...,n\}$ for $n \in \N$.
\\ \\
{\bf Notation}. We shall be manipulating sequences of data and so we describe here some notation that will be convenient. A sequence $(a_1,...,a_n)$ from some set will be denoted more tersely as $(a_i)$ leaving the length unmentioned. Similarly a sequence of sequences
\[ ((a_{11},...,a_{1n_1}),...,(a_{k1},...,a_{kn_k})) \]
of elements from some set will be denoted $(a_{ij})$ -- the variable $i$ ranges over $1{\leq}i{\leq}k$ and the variable $j$ ranges over $1{\leq}j{\leq}n_i$. Triply-nested sequences look like this $(a_{ijk})$, and so on. These conventions are more or less implicit already in the notation we have been using all along for multitensors. See especially section(\ref{ssec:LMC}) and \cite{EnHopI}. We denote by
\[ \con(a_{i_1,...,i_k}) \]
the ordinary sequence obtained from the $k$-tuply nested sequence $(a_{i_1,...,i_k})$ by concatenation. In particular given a sequence $(a_i)$, the set of $(a_{ij})$ such that $\con(a_{ij})=(a_i)$ is just the set of partitions of the original sequence into doubly-nested sequences, and will play an important role below. This is because to give the substitution maps for a multitensor $E$ on $V$, is to give maps
\[ \sigma : \opE\limits_i\opE\limits_j X_{ij} \to \opE\limits_i X_i \]
for all $(X_{ij})$ and $(X_i)$ from $V$ such that $\con(X_{ij})=(X_i)$.
\\ \\
\indent The monad map $\phi:M \to S$ is taken as $\Gamma(\psi):\ca GE_1 \to \Gamma(E)$ where $\psi:\widetilde{E_1} \hookrightarrow E$ is the inclusion of the unary part of the multitensor $E$. The role of $(X,x)$ in $V^M$ is played by sequences $(X_i,x_i)$ of $E_1$-algebras regarded as objects of $\ca G(V^{E_1})$. The transfinite induction produces for each ordinal $m$ and each sequence of $E_1$-algebras as above of length $n$, morphisms
\[ \begin{array}{l} {v^{(m)}_{(X_i,x_i)} : SQ_m(X_i,x_i) \rightarrow Q_{m{+}1}(X_i,x_i)} \\ {q^{(m)}_{(X_i,x_i)}:Q_m(X_i,x_i) \to Q_{m{+}1}(X_i,x_i)} \\ {q^{({<}m)}_{(X_i,x_i)}:S(X_i) \to Q_m(X_i,x_i)} \end{array} \]
in $\ca GV$ which are identities on objects, and thus we shall now evolve this notation so that it only records what's going on in the hom between $0$ and $n$. By the definition of $S$ we have the equation on the left
\[ \begin{array}{lccr} {S(X_i)(0,n) = \opE\limits_iX_i} &&& {Q_m(X_i,x_i)(0,n) = {\opEm\limits_i}(X_i,x_i)} \end{array} \]
and the equation on the right is a definition. Because of these definitions and that of $S$ we have the equation
\[ SQ_m(X_i,x_i)(0,n) = \coprod\limits_{\con(X_{ij},x_{ij}){=}(X_i,x_i)} \opE\limits_i\opEm\limits_j (X_{ij},x_{ij}). \]
The data for the hom maps of the $v^{(m)}$ thus consists of morphisms
\[ \begin{array}{c} {v^{(m)}_{(X_{ij},x_{ij})} : \opE\limits_i\opEm\limits_j (X_{ij},x_{ij}) \to \opEmpone\limits_i (X_i,x_i)} \end{array} \]
in $V$ whenever one has $\con(X_{ij},x_{ij})=(X_i,x_i)$ as sequences of $E_1$-algebras.

To summarise, the output of the transfinite process we are going to describe is, for each ordinal $m$, the following data. For each sequence $(X_i,x_i)$ of $E_1$-algebras, one has an object
\[ \opEm\limits_i (X_i,x_i) \]
and morphisms
\[ \begin{array}{l} {v^{(m)}_{(X_{ij},x_{ij})}} : {\opE\limits_i\opEm\limits_j (X_{ij},x_{ij}) \to \opEmpone\limits_i (X_i,x_i)} \\
{q^{(m)}_{(X_i,x_i)}} : {\opEm\limits_i (X_i,x_i) \to \opEmpone\limits_i (X_i,x_i)} \\
{q^{(<m)}_{(X_i,x_i)}} : {\opE\limits_iX_i \to \opEm\limits_i (X_i,x_i)} \end{array} \]
of $V$ where $\con(X_{ij},x_{ij})=(X_i,x_i)$.
\\ \\
{\bf Initial step}. First we put $\opEzero\limits_i (X_i,x_i) = \opE\limits_i X_i$, $q^{({<}0)}_{(X_i,x_i)_i} = \id$, and then form the coequaliser
\begin{equation}\label{eq:coeq}
\xygraph{!{0;(2,0):} {\opE\limits_iE_1X_i}="l" [r] {\opE\limits_iX_i}="m" [r] {\opEone\limits_i(X_i,x_i)}="r" "l":@<2ex>"m"^-{\sigma} "l":"m"_-{\opE\limits_ix_i}:@<1ex>"r"^-{q^{(0)}_{(X_i,x_i)}}} \end{equation}
in $V$ to define $q^{(0)}$. Put $v^{(0)}=q^{(0)}\sigma$ and $q^{({<}1)}=q^{(0)}$.
\\ \\
{\bf Inductive step}. Assuming that $v^{(m)}$, $q^{(m)}$ and $q^{({<}m{+}1)}$ are given, we have maps
\[ \begin{array}{lcr} {\xybox{\xygraph{!{0;(2,0):} {\opE\limits_i\opE\limits_j\opEm\limits_k}="l" [r] {\opE\limits_i\opEmpone\limits_{jk}}="r" "l":"r"^-{\opE\limits_iv^{(m)}}}}} && {\xybox{\xygraph{!{0;(2,0):} {\opE\limits_i\opE\limits_j\opEm\limits_k}="l" [r] {\opE\limits_{ij}\opEm\limits_k}="m" [r] {\opE\limits_{ij}\opEmpone\limits_k}="r" "l":"m"^-{{\sigma}\opEm\limits_k}:"r"^-{q^{(m)}}}}} \end{array} \]
and these are used to provide the parallel maps in the coequaliser
\[ \xygraph{!{0;(1.5,0):} {\coprod\limits_{\con(X_{ijk},x_{ijk})=(X_i,x_i)} \opE\limits_i\opE\limits_j\opEm\limits_k(X_{ijk},x_{ijk})}="l" [d] {\coprod\limits_{\con(X_{ij},x_{ij})=(X_i,x_i)} \opE\limits_i\opEmpone\limits_j(X_{ij},x_{ij})}="m" [d] {\opEmptwo\limits_i (X_i,x_i)}="r" "l":@<-2ex>"m" "l":@<2ex>"m":"r"^-{(v^{(m{+}1)}_{(X_{ij},x_{ij})})}} \]
which defines the $v^{(m{+}1)}$, the commutative diagram
\[ \xygraph{{\opEmpone\limits_i (X_i,x_i)}="l" [d] {E_1\opEmpone\limits_i (X_i,x_i)}="il" [r(4)] {\coprod\limits_{\con(X_{ij},x_{ij})=(X_i,x_i)} \opE\limits_i\opEmpone\limits_j(X_{ij},x_{ij})}="ir" [ru] {\opEmptwo\limits_i (X_i,x_i)}="r" "l":"il"_-{u}:"ir":"r"^(.35){v^{(m{+}1)}_{(X_i,x_i)}}:@{<-}"l"_-{q^{(m{+}1)}_{(X_i,x_i)}}} \]
in which the unlabelled map is the evident coproduct inclusion defines $q^{(m{+}1)}$, and $q^{({<}m{+}2)}=q^{(m{+}1)}q^{({<}m{+}1)}$.
\\ \\
{\bf Limit step}. Define $\opEm\limits_i (X_i,x_i)$ as the colimit of the sequence given by the objects $\opEr\limits_i (X_i,x_i)$ and morphisms $q^{(r)}$ for $r < m$, and $q_{<{m}}$ for the component of the universal cocone at $r=0$.
\[ \xygraph{!{0;(0,-1.5):(0,-3.45)::} {\colsum\limits_{\con(X_{ijk},x_{ijk})=(X_i,x_i)} \opE\limits_i\opE\limits_j\opEr\limits_k(X_{ijk},x_{ijk})}="tl" [r] {\colsum\limits_{\con(X_{ij},x_{ij})=(X_i,x_i)} \opE\limits_i\opEr\limits_j(X_{ij},x_{ij})}="tm" [r] {\colim_{r{<}m} \opEr\limits_i (X_i,x_i)}="tr" [d] {\opEm\limits_i (X_i,x_i)}="br" [l] {\coprod\limits_{\con(X_{ij},x_{ij})=(X_i,x_i)} \opE\limits_i\opEm\limits_j(X_{ij},x_{ij})}="bm" [l] {\coprod\limits_{\con(X_{ijk},x_{ijk})=(X_i,x_i)} \opE\limits_i\opE\limits_j\opEm\limits_k(X_{ijk},x_{ijk})}="bl" "tl":@<1ex>"tm"^-{\sigma^{(<{m})}}:@<1ex>"tr"^-{v^{(<{m})}} "tl":@<-1ex>"tm"_-{(Ev)^{(<{m})}}:@<-1ex>@{<-}"tr"_-{u^{(<{m})}} "bl":"bm"_-{\mu}:@{<-}"br"_-{uc} "tl":"bl"^{o_{m,2}} "tm":"bm"^{o_{m,1}} "tr":@{=}"br"} \]
As before we write $o_{m,1}$ and $o_{m,2}$ for the obstruction maps, and $c$ denotes the evident coproduct injection. The maps $\sigma^{(<{m})}$, $(Ev)^{(<{m})}$, $v^{(<{m})}$ and $u^{(<{m})}$ are by definition induced by $\sigma{\opEr}$, $(Ev)^{(r)}$, $v^{(r)}$ and $u{\opEr}$ for $r < m$ respectively. Define $v^{(m)}$ as the coequaliser of $o_{m,1}\sigma^{(<{m})}$ and $o_{m,1}(Ev)^{<{m}}$, $q^{(m)}=v^{(m)}(u{\opEm})$ and $q^{(<{m{+}1})}=q^{(m)}q^{(<{m})}$.
\\ \\
Instantiating corollary(\ref{cor:explicit-phi-shreik}) to the present situation gives
\begin{corollary}\label{cor:lifted-obj}
Let $V$ be a cocomplete category, $\lambda$ a regular cardinal, and $E$ a distributive $\lambda$-accessible multitensor on $V$. Then for any ordinal $m$ with $|m| \geq \lambda$ one may take
\[ (\opEm\limits_i (X_i,x_i), a(X_i,x_i)) \]
where the action $a(X_i,x_i)$ is given as the composite
\[ \xygraph{!{0;(3,0):} {E_1\opEm\limits_i (X_i,x_i)}="l" [r] {\opEmpone\limits_i (X_i,x_i)}="m" [r] {\opEm\limits_i (X_i,x_i)}="r" "l":"m"^-{v^{(m)}}:"r"^-{(q^{(m)})^{-1}}} \]
as an explicit description of the object map of the lifted multitensor $E'$ on $V^{E_1}$.
\end{corollary}
In corollaries (\ref{cor:phi-shreik-simple}) and (\ref{cor:vexp-simple}), in which the initial data is a monad map $\phi:M \to S$ between monads on a category $V$ together with an algebra $(X,x)$ for $M$, we noted the simplification of our constructions when $S$ and $S^2$ preserve the coequaliser
\begin{equation}\label{eq:monad-coeq} \xygraph{!{0;(2,0):} {SMX}="l" [r] {SX}="m" [r] {Q_1X}="r" "l":@<-1ex>"m"_-{Sx} "l":@<1ex>"m"^-{\mu^SS(\phi)}:"r"^-{q_0}} \end{equation}
in $V$, which is part of the first step of the inductive construction of $\phi_!$. In the present situation the role of $V$ is played by the category $\ca GV$, the role of $S$ is played by $\Gamma E$, and the role of $(X,x)$ played by a given sequence $(X_i,x_i)$ of $E_1$-algebras, and so the role of the coequaliser (\ref{eq:monad-coeq}) is now played by the coequaliser
\begin{equation}\label{eq:monad-coeq2} \xygraph{!{0;(2,0):} {\Gamma E(E_1X_i)}="l" [r(1.5)] {\Gamma E(X_i)}="m" [r] {Q_1}="r" "l":@<-1ex>"m"_-{Sx} "l":@<1ex>"m"^-{\mu^SS(\phi)}:"r"^-{q^{(0)}}} \end{equation}
in $\ca GV$. Here we have denoted by $Q_1$ the $V$-graph with objects $\{0,...,n\}$ and homs given by
\[ Q_1(i,j) = \left\{\begin{array}{lll} {\emptyset} && {\textnormal{if $i>j$}} \\ {\opEone\limits_{i{<}k{\leq}j}(X_k,x_k)} && {\textnormal{if $i \leq j$.}} \end{array}\right. \]
Taking the hom of (\ref{eq:monad-coeq2}) between $0$ and $n$ gives the coequaliser
\begin{equation}\label{eq:mult-coeq} \xygraph{!{0;(2,0):} {\opE\limits_iE_1X_i}="l" [r] {\opE\limits_iX_i}="m" [r] {\opEone\limits_i(X_i,x_i)}="r" "l":"m"_-{\opE\limits_ix_i} "l":@<2ex>"m"^-{\sigma}:@<1ex>"r"^-{q^{(0)}}} \end{equation}
in $V$ which is part of the first step of the explicit inductive construction of $E'$. We shall refer to (\ref{eq:mult-coeq}) as the \emph{basic coequaliser associated to the sequence $(X_i,x_i)$} of $E_1$-algebras. Note that all coequalisers under discussion here are reflexive coequalisers, with the common section for the basic coequalisers given by the maps $\opE\limits_iu_{X_i}$.

The basic result which expresses why reflexive coequalisers are nice, is the $3{\times}3$-lemma, which we record here for the reader's convenience. A proof can be found in \cite{PTJ-topos77}.
\begin{lemma}\label{lem:3by3}
{\bf $3{\times}3$-lemma}. Given a diagram
\[ \xymatrix @R=3em @C=3em {A \ar@<1ex>[r]^-{f_1} \ar@<-1ex>[r]_{g_1} \ar@<1ex>[d]^{b_1} \ar@<-1ex>[d]_{a_1}
& B \ar[r]^-{h_1} \ar@<1ex>[d]^{b_2} \ar@<-1ex>[d]_{a_2} & C \ar@<1ex>[d]^{b_3} \ar@<-1ex>[d]_{a_3} \\
D \ar@<1ex>[r]^-{f_2} \ar@<-1ex>[r]_-{g_2} & E \ar[r]^-{h_2} & F \ar[d]^{c} \\ && H} \]
in a category such that: (1) the two top rows and the right-most column are coequalisers, (2) $a_1$ and $b_1$ have a common section, (3) $f_1$ and $g_1$ have a common section, (3) $f_2a_1{=}a_2f_1$, (4) $g_2b_1{=}b_2g_1$, (5) $h_2a_2{=}a_3h_1$ and (6) $h_2b_2{=}b_3h_1$; then $ch_2$ is a coequaliser of $f_2a_1{=}a_2f_1$ and $g_2b_1{=}b_2g_1$.
\end{lemma}
\noindent If $F:{\ca A_1}{\times}...{\times}{\ca A_n}{\rightarrow}{\ca B}$ is a functor which preserves connected colimits of a certain type, then it also preserves these colimits in each variable separately, because for a connected colimit, a cocone involving only identity arrows is a universal cocone. The most basic corollary of the $3{\times}3$-lemma says that the converse of this is true for reflexive coequalisers.
\begin{corollary}\label{cor:3by3}
Let $F:{\ca A_1}{\times}...{\times}{\ca A_n}{\rightarrow}{\ca B}$ be a functor. If $F$ preserves reflexive coequalisers in each variable separately then $F$ preserves reflexive coequalisers.
\end{corollary}
\noindent and this can be proved by induction on $n$ using the $3{\times}3$-lemma in much the same way as \cite{LkMonFinMon} lemma(1). The most well-known instance of this is
\begin{corollary}\label{cor:3by3-2}\cite{LkMonFinMon}
Let $\ca V$ be a biclosed monoidal category. Then the $n$-fold tensor product of reflexive coequalisers in $\ca V$ is again a reflexive coequaliser.
\end{corollary}
\noindent In particular note that by corollary(\ref{cor:3by3}) a multitensor $E$ preserves (some class of) reflexive coequalisers iff it preserves them in each variable separately.

Returning to our basic coequalisers an immediate consequence of the explicit description of $\Gamma E$ and corollary(\ref{cor:3by3}) is
\begin{lemma}\label{lem:reformulate-simplifying-conditions}
Let $E$ be a distributive multitensor on $V$ a cocomplete category, and $(X_i,x_i)$ a sequence of $E_1$-algebras. If $E$ preserves the basic coequalisers associated to all the subsequences of $(X_i,x_i)$, then for all $r \in \N$, $(\Gamma E)^r$ preserves the coequaliser (\ref{eq:monad-coeq2}).
\end{lemma}
\noindent and applying this lemma and corollary(\ref{cor:phi-shreik-simple}) gives
\begin{corollary}\label{cor:lifted-obj-simple}
Let $V$ be a cocomplete category, $\lambda$ a regular cardinal, $E$ a distributive $\lambda$-accessible multitensor on $V$ and $(X_i,x_i)$ a sequence of $E_1$-algebras. If $E$ preserves the basic coequalisers associated to all the subsequences of $(X_i,x_i)$, then one may take
\[ \opEpr\limits_i(X_i,x_i) = (\opEone\limits_i (X_i,x_i), a) \]
where the action $a$ is defined as the unique map such that $aE_1(q^{(0)})=q^{(0)}\sigma$.
\end{corollary}
\noindent Note in particular that when the sequence $(X_i,x_i)$ of $E_1$-algebras is of length $n=0$ or $n=1$, the associated basic coequaliser is absolute. In the $n=0$ case the basic coequaliser is constant at $E_0$, and when $n=1$ the basic coequaliser may be taken to be the canonical presentation of the given $E_1$-algebra. Thus in these cases it follows from corollary(\ref{cor:lifted-obj-simple}) that $E'_0=(E_0,\sigma)$ and $E_1'(X,x)=(X,x)$. Reformulating the explicit description of the unit in corollary(\ref{cor:vexp-simple}) one recovers the fact from our explicit descriptions, that the unit of $E'$ is the identity, which was of course true by construction.

To complete the task of giving a completely explicit description of the multitensor $E'$ we now turn to unpacking its substitution. So we assume that $E$ is a distributive $\lambda$-accessible multitensor on $V$ a cocomplete category, and fix an ordinal $m$ so that $|m| \geq \lambda$, so that $E'$ may be constructed as $E^{(m)}$ as in corollary(\ref{cor:lifted-obj}). By transfinite induction on $r$ we shall generate the following data:
\[ \sigma^{(r)}_{X_{ij},x_{ij}} : \opEr\limits_i(\opEm\limits_j(X_{ij}),x_{ij}) \to \opEm\limits_i(X_i,x_i) \]
and $\sigma^{(r{+}1)}_{X_{ij},x_{ij}}$ whenever $\con(X_{ij},x_{ij})=(X_i,x_i)$, such that
\[ \xygraph{!{0;(2.5,0):(0,.5)::} {\opE\limits_i\opEr\limits_j\opEm\limits_k}="tl" [r] {\opErpone\limits_{ij}\opEm\limits_k}="tr" [d] {\opEm\limits_{ijk}}="br" [l] {\opE\limits_i\opEm\limits_{jk}}="bl" "tl":"tr"^-{v^{(r)}E^{(m)}}:"br"^-{\sigma^{(r{+}1)}}:@{<-}"bl"^-{(q^{(m)})^{-1}v^{(m)}}:@{<-}"tl"^-{\opE\limits_i\sigma^{(r)}}} \]
commutes.
\\ \\
{\bf Initial step}. Define $\sigma^{(0)}$ to be the identity and $\sigma^{(1)}$ as the unique map such that $\sigma^{(1)}q^{(0)}=(q^{(m)})^{-1}v^{(m)}$ by the universal property of the coequaliser $q^{(0)}$.
\\ \\
{\bf Inductive step}. Define $\sigma^{(r{+}2)}$ as the unique map such that
\[ \sigma^{(r{+}2)}(v^{(r{+}1)}E^{(m)})=(q^{(m)})^{-1}v^{(m)}(\opE\limits_i\sigma^{(r{+}1)}) \]
using the universal property of $v^{(r{+}1)}$ as a coequaliser.
\\ \\
{\bf Limit step}. When $r$ is a limit ordinal define $\sigma^{(r)}$ as induced by the $\mu^{(s)}$ for $s<r$ and the universal property of $E^{(r)}$ as the colimit of the sequence of the $E^{(s)}$ for $s<r$. Then define $\sigma^{(r{+}1)}$ as the unique map such that
\[ \sigma^{(r{+}1)}(v^{(r)}E^{(m)})=(q^{(m)})^{-1}v^{(m)}(\opE\limits_i\sigma^{(r)}) \]
using the universal property of $v^{(r)}$ as a coequaliser.
\\ \\
The fact that the transfinite construction just specified was obtained from that for corollary(\ref{cor:induced-monad-very-explicit}), by taking $S=\Gamma E$ and looking at the homs, means that by corollaries (\ref{cor:induced-monad-very-explicit}) and (\ref{cor:vexp-simple}) one has
\begin{corollary}\label{cor:induced-substitution-very-explicit}
Let $V$ be a cocomplete category, $\lambda$ a regular cardinal, $E$ a distributive $\lambda$-accessible multitensor on $V$ and $(X_i,x_i)$ a sequence of $E_1$-algebras. Then one has
\[ \sigma'_{(X_i,x_i)} = \sigma^{(m)}_{(X_i,x_i)} \]
as an explicit description of the substitution of $E'$. If moreover $E$ preserves the basic coequalisers of all the subsequences of $(X_i,x_i)$, then one may take $\sigma^{(1)}_{(X_i,x_i)}$ as the explicit description of the substitution.
\end{corollary}
%

\section{Gray and Crans tensor products}\label{sec:Gray-Crans}

Let $A$ be a $\ca T_{\leq{n{+}1}}$-operad over $\Set$ and let $E$ be the associated $\ca T_{\leq{n}}^{\times}$-multitensor. Thus by definition one has
\[ \begin{array}{lccr} {A = \Gamma E} &&& {E = \overline{A}} \end{array} \]
and $\ca G^{n+1}(\Set)^A \iso \Enrich E$ over $\ca G^{n+1}(\Set)$. The monad $E_1$ on $\ca G^n{\Set}$ has as algebras, the structure borne by the homs of an $A$-algebra. Theorem(\ref{thm:lift-mult}) produces the functor operad $E'$ on $\ca G^n(\Set)^{E_1}$ such that
\[ \ca G^{n{+}1}(\Set)^A \iso \Enrich {E'} \iso \Enrich {E} \]
over $\ca G^n(\Set)^{E_1}$. Moreover $E'$ is the unique such functor operad which is distributive.
\begin{example}\label{ex:strict-n-cat}
When $A$ is the terminal $\ca T_{\leq{n{+}1}}$-operad, $E$ is the terminal $\ca T_{\leq{n}}^{\times}$-multitensor, and so $E_1 = \ca T_{\leq{n}}$. Since strict $(n{+}1)$-categories are categories enriched in $\Enrich n$ using cartesian products, and these commute with coproducts (in fact all colimits), it follows by the uniqueness part of theorem(\ref{thm:lift-mult}) that $E'$ is just the cartesian product of $n$-categories. 
\end{example}
\begin{example}\label{ex:opmonoidal}
Suppose that $E$ is a multitensor on $V$ and $T$ is an opmonoidal monad on $(V,E)$. Then one has by theorem(49) of \cite{WebMultMnd} a lifted multitensor $E'$ on $V^T$. On the other hand if moreover $V$ is cocomplete, $E$ is a distributive and accessible functor operad, and $T$ is coproduct preserving and accessible, then $E'$ may also be obtained by applying theorem(\ref{thm:lift-mult}) to the composite multitensor $EM(T)$. When $E$ is given by cartesian product and $T = \ca T_{{\leq}n}$, we recover example(\ref{ex:strict-n-cat}).
\end{example}
\begin{example}\label{ex:Gray}
Take $A$ to be the $\ca T_{\leq{3}}$-operad for Gray categories constructed in \cite{Bat98} (example(4) after corollary(8.1.1)). Since $E_1$ is the monad on $\ca G^2(\Set)$ for 2-categories, in this case $E'$ is a functor operad for 2-categories. However the Gray tensor product of 2-categories \cite{Gray} is part of a symmetric monoidal closed structure. Thus it is distributive as a functor operad, and since Gray categories are categories enriched in the Gray tensor product by definition, it follows that $E'$ is the Gray tensor product. In other words, the general methods of this paper have succeeded in producing the Gray tensor product of $2$-categories from the operad $A$.
\end{example}
\begin{example}\label{ex:Crans}
In \cite{Crans99} Sjoerd Crans explicitly constructed a tensor product on the category of Gray-categories. This explicit construction was extremely complicated. It is possible to exhibit the Crans tensor product as an instance of our general theory, by rewriting his explicit constructions as the construction of the $\ca T_{{\leq}4}$-operad $A$ whose algebras are teisi in his sense. The associated multitensor $E$ has $E_1$ equal to the $\ca T_{{\leq}3}$-operad for Gray categories. Thus theorem(\ref{thm:lift-mult}) constructs a functor operad $E'$ of Gray categories whose enriched categories are teisi. Since the tensor product explicitly constructed by Crans is distributive, the uniqueness of part of theorem(\ref{thm:lift-mult}) ensures that it is indeed $E'$, since teisi are categories enriched in the Crans tensor product by definition.
\end{example}
Honestly writing the details of the $\ca T_{{\leq}4}$-operad of example(\ref{ex:Crans}) is a formidable task and we have omitted this here. In the end though, such details will not be important, because such a tensor product (or more properly a biclosed version thereof) will only be really useful once it is constructed in a conceptual way as part of a general inductive machine.

\section{Contractibility}\label{sec:contractibility}

\subsection{Functoriality and comparison}\label{ssec:functoriality-lifting}
Recall \cite{Str72} \cite{LS00} that when a 2-category $\ca K$ has Eilenberg-Moore objects, the one and 2-cells of the 2-category $\MND(\ca K)$ admit another description. Given monads $(V,T)$ and $(W,S)$ in $\ca K$, to give a monad functor $(H,\psi):(V,T) \rightarrow (W,S)$, is to give $\tilde{H}:V^T \rightarrow W^S$ such that $U^S\tilde{H}=HU^T$. This follows immediately from the universal property of Eilenberg-Moore objects. Similarly to give a monad 2-cell $\phi:(H_1,\psi_1) \rightarrow (H_2,\psi_2)$ is to give $\phi:H_1{\rightarrow}H_2$ and $\tilde{\phi}:\tilde{H_1}{\rightarrow}\tilde{H_2}$ commuting with $U^T$ and $U^S$. Note that Eilenberg-Moore objects in $\CAT/\Set$ are computed as in $\CAT$, and we shall soon apply these observations to the case $\ca K = \CAT/\Set$.
\begin{remark}\label{rem:functoriality-lifting}
Suppose we have a lax monoidal functor $(H,\psi):(V,E) \rightarrow (W,F)$. Then we obtain a commutative diagram
\[ \xygraph{!{0;(1.5,0):(0,.667)::} {\Enrich E}="tl" [r] {\ca G(V^{E_1})}="tm" [r] {\ca GV}="tr" [d] {\ca GW}="br" [l] {\ca G(W^{F_1})}="bm" [l] {\Enrich F}="bl" "tl":"tm":"tr" "bl":"bm":"br" "tl":"bl" "tm":"bm" "tr":"br"} \]
of forgetful functors in $\CAT/\Set$. If moreover $V$ and $W$ are cocomplete and $E$ and $F$ are distributive and accessible, then by theorem(\ref{thm:lift-mult}) we have distributive multitensors $E'$ and $F'$ on $V^{E_1}$ and $V^{F_1}$ respectively, and from the left-most square above we have a monad morphism $(\ca GV,\Gamma{E'}) \rightarrow (\ca GW,\Gamma{F'})$ with underlying functor $\ca G(\psi_1^*)$. By \cite{WebMultMnd} proposition(44) this monad functor is the result of applying $\Gamma$ to a unique lax monoidal functor
\[ (\psi_1^*,\psi') : (V^{E_1},E'){\rightarrow}(W^{F_1},F'). \]
Arguing similarly for monoidal transformations and monad 2-cells, one finds that the assignment $(V,E) \mapsto (V^{E_1},E')$, for cocomplete $V$ and accessible $E$, is 2-functorial.
\end{remark}
\begin{remark}\label{rem:general-comparison}
Suppose that $\varepsilon:E \rightarrow \ca T^{\times}_{{\leq}n}$ is a $\ca T^{\times}_{{\leq}n}$-multitensor. Applying remark(\ref{rem:functoriality-lifting}) in the case $V=W=\ca G^n\Set$, $H=\id$, $\psi{=}\varepsilon$ and example(\ref{ex:strict-n-cat}) one obtains a map
\[ \varepsilon'_{(X_i,x_i)} : {\opE\limits_i}' \varepsilon_1^*(X_i,x_i) \rightarrow \prod\limits_i \varepsilon_1^*(X_i,x_i) \]
of $E_1$-algebras for each sequence $((X_1,x_1),...,(X_n,x_n))$ of strict $n$-categories, since $\varepsilon_1^*:\Enrich n \rightarrow \Alg {E_1}$ as a right adjoint preserves products. This gives a general comparison map between the functor operad $E'$ produced by theorem(\ref{thm:lift-mult}) and cartesian products, defined for sequences of $E_1$-algebras that underlie strict $n$-categories.
\end{remark}
\begin{example}\label{ex:Gray-comparison}
When $E$ is the multitensor of example(\ref{ex:Gray}) for Gray categories, $E_1$ is itself $\ca T_{{\leq}2}$ and $\varepsilon_1{=}\id$, and $\varepsilon'$ gives the well-known comparison map from the Gray tensor product of 2-categories to the cartesian product, which we recall is actually a componentwise biequivalence.
\end{example}
Returning to the situation of remark(\ref{rem:functoriality-lifting}), it is routine to unpack the assignment $(H,\psi) \mapsto (\psi_1^*,\psi')$ as in section(\ref{sec:explicit-lifting}) and so obtain the following 1-cell counterpart of corollary(\ref{cor:lifted-obj-simple}).
\begin{corollary}\label{cor:free-lift-1cell}
Let $(H,\psi):(V,E){\rightarrow}(W,F)$ be a lax monoidal functor such that $V$ and $W$ are cocomplete, and $E$ and $F$ are accessible. Let $(X_1,...,X_n)$ be a sequence of objects of $V$. Then the component of $\psi'$ at the sequence \[ (E_1X_1,...,E_1X_n) \] of free $E_1$-algebras is just $\psi_{X_i}$.
\end{corollary}
%

\subsection{Contractible multitensors revisited}\label{sec:contractibility-lifted-multitensors}
In section(7.3) of \cite{WebMultMnd} we saw that a $\ca T_{\leq{n{+}1}}$-operad $A$ over $\Set$ is contractible iff its associated $\ca T_{\leq{n}}^{\times}$-multitensor $E$ is contractible. We now extend this result to the associated functor operad $E'$ on $\ca G^n(\Set)^{E_1}$.
\begin{proposition}\label{prop:contractible}
Let $(H,\psi):(V,E){\rightarrow}(W,F)$ be a lax monoidal functor between distributive lax monoidal categories, and $\ca I$ a class of maps in $W$. Suppose that $W$ is extensive, $H$ preserves coproducts and the codomains of maps in $\ca I$ are connected. Then the following statements are equivalent
\begin{itemize}
\item[(1)] $\psi$ is a trivial $\ca I$-fibration.
\item[(2)] $\Gamma\psi$ is a trivial $\ca I^+$-fibration.
\end{itemize}
and moreover when in addition $V$ and $W$ are cocomplete and $E$ and $F$ are accessible, these conditions are also equivalent to
\begin{itemize}
\item[(3)] The components of $U^F\psi'$ at sequences $(E_1X_1,...,E_1X_n)$ of free $E_1$-algebras are trivial $\ca I$-fibrations.
\end{itemize}
\end{proposition}
\begin{proof}
$(1){\iff}(2)$ is proposition(58) of \cite{WebMultMnd} and $(2){\iff}(3)$ follows immediately from corollary(\ref{cor:free-lift-1cell}).
\end{proof}
Recall from \cite{WebMultMnd} section(7.2) that $\ca I_{{\leq}n}$ is the set of boundary inclusions of $m$-globes for $m \leq n$. Then proposition(\ref{prop:contractible}) has the immediate
\begin{corollary}\label{cor:contractible}
Let $0 \leq n \leq \infty$, $\alpha:A \rightarrow \ca T_{{\leq}n{+}1}$ be an $n{+}1$-operad over $\Set$ and $\varepsilon:E \rightarrow \ca T^{\times}_{{\leq}n}$ be the corresponding $n$-multitensor. TFSAE:
\begin{enumerate}
\item  $\alpha:A \rightarrow \ca T_{{\leq}n{+}1}$ is contractible.\label{cont1}
\item  $\varepsilon:E \rightarrow \ca T^{\times}_{{\leq}n}$ is contractible.\label{cont2}
\item  The components of $\varepsilon'_{(X_i,x_i)}$ of remark(\ref{rem:general-comparison}) are trivial $\ca I_{{\leq}n}$-fibrations of $n$-globular sets, when the $(X_i,x_i)$ are free strict $n$-categories.
\end{enumerate}
\end{corollary}
The equivalence of (\ref{cont1}) and (\ref{cont2}) appeared already as \cite{WebMultMnd} corollary(59).
\begin{example}\label{ex:Gray-contractible}
When $\alpha$ is the $\ca T_{{\leq}3}$-operad for Gray-categories, the contractibility of $\alpha$ is a consequence of the fact that the canonical 2-functors from the Gray to the cartesian tensor product are identity-on-object biequivalences.
\end{example}

\section{Standard convolution}\label{sec:convolution}

\subsection{Recalling convolution for multicategories}\label{sec:substitudes}
The set of multimaps $(X_1,...,X_n) \to Y$ in a given multicategory $\C$ shall be denoted as $\C(X_1,...,X_n;Y)$. Recall that a linear map in $\C$ is a multimap whose domain is a sequence of length $1$. The objects of $\C$ and linear maps between them form a category, which we denote as $\C_l$, and we call this the linear part of $\C$. The set of objects of $\C$ is denoted as $\C_0$. Given objects
\[ \begin{array}{lcccr} {A_{11}, ..., A_{1n_1}, ... ..., A_{k1},...,A_{kn_k}} && {B_1,...,B_k} && {C} \end{array} \]
of $\C$, we denote by
\[ {\sigma_{A,B,C} : \C(B_1,...,B_k;C) \times \prod\limits_i \C(A_{i1},...,A_{in_i};B_i) \rightarrow \C(A_{11},...,A_{kn_k};C)} \]
the substitution functions of the multicategory $\C$. One thus induces a function
\[ \sigma_{A,C} : \int^{B_1,...,B_k} \C(B_1,...,B_k;C) \times \prod\limits_i \C(A_{i1},...,A_{in_i};B_i) \rightarrow \C(A_{11},...,A_{kn_k};C) \]
in which for the purposes of making sense of this coend, the objects $B_1,...,B_k$ are regarded as objects of the category $\C_{l}$. A \emph{promonoidal category} in the sense of Day \cite{DayConvolution}, in the unenriched context, can be defined as a multicategory $\C$ such that these induced functions $\sigma_{A,C}$ are all bijective. A \emph{promonoidal structure} on a category $\D$ is a promonoidal category $\C$ such that $\C_l = \D^{\op}$.

A lax monoidal category $(V,E)$ is \emph{cocomplete} when $V$ is cocomplete as a category and $E_n:V^n \to V$ preserves colimits in each variable for all $n \in \N$. In this situation the multitensor $E$ is also said to be cocomplete. When $\C$ is small it defines a functor operad on the functor category $[\C_l,\Set]$ whose tensor product $F$ is given by the coend
\[ \opF\limits_i X_i = \int^{C_1,...,C_n} \C(C_1,...,C_n;-) \times \prod\limits_i X_iC_i \]
and substitution is defined in the evident way from that of $\C$. By proposition(2.1) of \cite{DS-Substitude} $F$ is a cocomplete functor operad and is called the \emph{standard convolution} structure of $\C$ on $[\C_l,\Set]$. By proposition(2.2) of \cite{DS-Substitude}, for each fixed category $\D$, standard convolution gives an equivalence between multicategories on $\C$ such that $\C_l = \D$ and cocomplete functor operads on $[\D,\Set]$, which restricts to the well-known \cite{DayConvolution} equivalence between promonoidal structures on $\D^{\op}$ and closed monoidal structures on $[\D,\Set]$.

We have recalled these facts in a very special case compared with the generality at which this theory is developed in \cite{DS-Substitude}. In that work all structures are considered as enriched over some nice symmetric monoidal closed base $\ca V$, and moreover rather than $\D = \C_l$ as above, one has instead an identity on objects functor $\D \to \C_l$. The resulting combined setting is then what are called $\ca V$-substitudes in \cite{DS-Substitude}, and in the $\ca V = \Set$ case the extra generality of the functor $\D \to \C_l$, corresponds at the level of multitensors, to the consideration of general closed multitensors on $[\D,\Set]$ instead of mere functor operads. In this section we shall recover standard convolution, for the special case that we have described above, from the lifting theorem.

\subsection{Convolution via lifting}\label{sec:convolution-via-lifting}
Given a multicategory $\C$ we define the multitensor $E$ on $[\C_0,\Set]$ via the formula
\[ \left(\opE\limits_{1{\leq}i{\leq}n} X_i\right)(C) = \coprod\limits_{C_1,...,C_n} \left(\C(C_1,...,C_n;C) \times \prod\limits_{1{\leq}i{\leq}n} X_i(C_i)\right) \]
using the unit and compositions for $\C$ in the evident way to give the unit $u$ and substitution $\sigma$ for $E$. When $\C_0$ has only one element, this is the multitensor on $\Set$ coming from the operad $P$ described in \cite{EnHopI} and \cite{WebMultMnd}, whose tensor product is given by the formula
\[ \opE\limits_{1{\leq}i{\leq}n} X_i = P_n \times X_1 \times ... \times X_n.  \]
An $E$-category with one object is exactly an algebra of the coloured operad $P$ in the usual sense. A general $E$-category amounts to a set $X_0$, sets $X(x_1,x_2)(C)$ for all $x_1,x_2 \in X_0$ and $C \in \C_0$, and functions
\begin{equation}\label{eq:sc-cat} \C(C_1,...,C_n;C) \times \prod\limits_i X(x_{i-1},x_i)(C_i) \to X(x_0,x_n)(C) \end{equation}
compatible in the evident way with the multicategory structure of $\C$. On the other hand an $F$-category amounts to a set $X_0$, sets $X(x_1,x_2)(C)$ natural in $C$, and maps as in (\ref{eq:sc-cat}) but which are natural in $C_1,...,C_n,C$, and compatible with $\C$'s multicategory structure. However this added naturality enjoyed by an $F$-category isn't really an additional condition, because it follows from the compatibility with the linear maps of $\C$. Thus $E$ and $F$-categories coincide, and one may easily extend this to functors and so give $\Enrich E \iso \Enrich F$ over $\ca G[\C_0,\Set]$.

The unary part of $E$ is given on objects by
\[ E_1(X)(C) = \coprod\limits_{D} \C_l(D,C) \times X(D) \]
which should be familiar -- $E_1$ is the monad on $[\C_0,\Set]$ whose algebras are functors $\C_l \to \Set$, and may be recovered from left kan extension and restriction along the inclusion of objects $\C_0 \hookrightarrow \C_l$. Thus the category of algebras of $E_1$ may be identified with the functor category $[\C_l,\Set]$. Since the multitensor $E$ is clearly cocomplete, it satisfies the hypotheses of theorem(\ref{thm:lift-mult}), and so one has a unique finitary distributive multitensor $E'$ on $[\C_l,\Set]$ such that $\Enrich E \iso \Enrich {E'}$ over $\ca G[\C_0,\Set]$. By uniqueness we have
\begin{proposition}\label{prop:convolution-via-lifting}
Let $\C$ be a multicategory, $F$ be the standard convolution structure on $[\C_l,\Set]$ and $E$ be the multitensor on $[\C_0,\Set]$ defined above. Then one has an isomorphism $F \iso E'$ of multitensors.
\end{proposition}
In particular when $\C$ is a promonoidal category proposition(\ref{prop:convolution-via-lifting}) expresses classical Day convolution, in the unenriched context, as a lift in the sense of theorem(\ref{thm:lift-mult}).

\appendix

\section{Transfinite constructions in monad theory}\label{sec:transfinite}

\subsection{Overview}\label{ssec:transfinite-overview}
Here we review some of the transfinite constructions in monad theory that we will use in sections(\ref{sec:lifting-theorem}) and (\ref{sec:explicit-lifting}). An earlier reference for these matters is \cite{Kel80}. However due to the technical nature of this material, and our need for its details when we come to making our constructions explicit, we feel that it is appropriate to give a rather thorough account of this background.

\subsection{Coequalisers in categories of algebras}\label{ssec:coequalisers}
Let $T$ be a monad on a category $V$ that has filtered colimits and coequalisers and let
\[ \xygraph{!{0;(2,0):} {(A,a)}="l" [r] {(B,b)}="r" "l":@<-1ex>"r"_-{g} "l":@<1ex>"r"^-{f}} \]
be morphisms in $V^T$. We shall now construct morphisms
\[ \begin{array}{lcccr} {v_n : TQ_n \rightarrow Q_{n{+}1}} && {q_n:Q_n \to Q_{n{+}1}} && {q_{{<}n}:B \to Q_n} \end{array} \]
starting with $Q_0=B$ by transfinite induction on $n$, such that for $n$ large enough $q_{<{n}}$ is the coequaliser of $f$ and $g$ in $V^T$ when $T$ is accessible. The initial stages of this construction are described in the following diagram.
\[ \xygraph{!{0;(1.5,0):(0,.666)::} {T^2A}="tA" [r] {T^2B}="tB" [r] {T^2Q_1}="t0" [r] {T^2Q_2}="t1" [r] {T^2Q_3}="t2" [r] {T^2Q_4}="t3" [r] {...}="t4"
"tA" [d] {TA}="mA" [r] {TB}="mB" [r] {TQ_1}="m0" [r] {TQ_2}="m1" [r] {TQ_3}="m2" [r] {TQ_4}="m3" [r] {...}="m4"
"mA" [d] {A}="bA" [r] {B}="bB" [r] {Q_1}="b0" [r] {Q_2}="b1" [r] {Q_3}="b2" [r] {Q_4}="b3" [r] {...}="b4"
"tA":@<-1ex>"tB"_-{T^2g} "tA":@<1ex>"tB"^-{T^2f}:"t0"^-{T^2q_0}:"t1"^-{T^2q_1}:"t2"^-{T^2q_2}:"t3"^-{T^2q_3}:"t4"
"mA":@<-1ex>"mB"_-{Tg} "mA":@<1ex>"mB"^-{Tf}:"m0"^-{Tq_0}:"m1"^-{Tq_1}:"m2"^-{Tq_2}:"m3"^-{Tq_3}:"m4"
"bA":@<-1ex>"bB"_-{g} "bA":@<1ex>"bB"^-{f}:"b0"_-{q_0}:"b1"_-{q_1}:"b2"_-{q_2}:"b3"_-{q_3}:"b4"
"tA":@<-1ex>"mA"_-{\mu}:@<-1ex>"bA"_-{a} "tA":@<1ex>"mA"^-{Ta}:@<1ex>@{<-}"bA"^-{\eta}
"tB":@<-1ex>"mB"_-{\mu}:@<-1ex>"bB"_-{b} "tB":@<1ex>"mB"^-{Tb}:@<1ex>@{<-}"bB"^-{\eta}
"t0":"m0"^-{\mu}:@{<-}"b0"^-{\eta} "t1":"m1"^-{\mu}:@{<-}"b1"^-{\eta} "t2":"m2"^-{\mu}:@{<-}"b2"^-{\eta} "t3":"m3"^-{\mu}:@{<-}"b3"^-{\eta}
"tB":"m0"^-{Tv_0} "mB":"b0"^-{v_0} "t0":"m1"^-{Tv_1} "m0":"b1"^-{v_1} "t1":"m2"^-{Tv_2} "m1":"b2"^-{v_2} "t2":"m3"^-{Tv_3} "m2":"b3"^-{v_3}} \]
{\bf Initial step}. Define $q_{{<}0}$ to be the identity, $q_0$ to be the coequaliser of $f$ and $g$, $q_{<{1}}=q_0$ and $v_0=q_0b$. Note also that $q_0=v_0\eta_B$.
\\ \\
{\bf Inductive step}. Assuming that $v_n$, $q_n$ and $q_{<{n{+}1}}$ are given, we define $v_{n{+}1}$ to be the coequaliser of $T(q_n)\mu$ and $Tv_n$, $q_{n{+}1}=v_{n{+}1}\eta$ and $q_{<{n{+}2}}=q_{n{+}1}q_{<{n{+}1}}$. One may easily verify that $q_{n{+}1}v_n=v_{n{+}1}T(q_n)$, and that $v_1$ could equally well have been defined as the coequaliser of $\eta{v_0}$ and $Tq_0$.
\\ \\
{\bf Limit step}. Define $Q_n$ as the colimit of the sequence given by the objects $Q_m$ and morphisms $q_m$ for $m < n$, and $q_{<{n}}$ for the component of the universal cocone at $m=0$.
\[ \xygraph{!{0;(3,0):(0,.333)::} {\colim_{m{<}n} T^2Q_m}="tl" [r] {\colim_{m{<}n} TQ_m}="tm" [r] {\colim_{m{<}n} Q_m}="tr" [d] {Q_n}="br" [l] {TQ_n}="bm" [l] {T^2Q_n}="bl" "tl":@<1ex>"tm"^-{\mu_{<{n}}}:@<1ex>"tr"^-{v_{<{n}}} "tl":@<-1ex>"tm"_-{(Tv)_{<{n}}}:@<-1ex>@{<-}"tr"_-{\eta_{<{n}}} "bl":"bm"_-{\mu}:@{<-}"br"_-{\eta} "tl":"bl"_{o_{n,2}} "tm":"bm"^{o_{n,1}} "tr":@{=}"br"} \]
We write $o_{n,1}$ and $o_{n,2}$ for the obstruction maps measuring the extent to which $T$ and $T^2$ preserve the colimit defining $Q_n$. We write $\mu_{<{n}}$, $(Tv)_{<{n}}$, $v_{<{n}}$ and $\eta_{<{n}}$ for the maps induced by the $\mu_{Q_m}$, $Tv_m$, $v_m$ and $\eta_{Q_m}$ for $m < n$ respectively. The equations
\[ \begin{array}{ccccccc} {{\mu}o_{n,2}=o_{n,1}\mu_{<{n}}} && {\eta=o_{n,1}\eta_{<{n}}} && {v_{<{n}}(Tv)_{<{n}}=v_{<{n}}\mu_{<{n}}} && {v_{<{n}}\eta_{<{n}}=\id} \end{array} \]
follow easily from the definitions. Define $v_n$ as the coequaliser of $o_{n,1}\mu_{<{n}}$ and $o_{n,1}(Tv)_{<{n}}$, $q_n=v_n\eta$ and $q_{<{n{+}1}}=q_nq_{<{n}}$.
\\ \\
{\bf Stabilisation}. We say that the sequence \emph{stabilises at $n$} when $q_n$ and $q_{n{+}1}$ are isomorphisms. In the case $n=0$ one may easily show that stabilisation is equivalent to just $q_0$ being an isomorphism, which is the same as saying that $f=g$.
\begin{lemma}\label{stable-limit}
If $n$ is a limit ordinal and $o_{n,1}$ and $o_{n,2}$ are invertible, then the sequence stabilises at $n$.
\end{lemma}
\begin{proof}
Let us write $q_{m,n}:Q_m \to Q_n$,
\[ \begin{array}{lccr} {q'_{m,n}:TQ_m \to \colim_{m<n}TQ_m} &&& {q''_{m,n}:T^2Q_m \to \colim_{m<n}T^2Q_m} \end{array} \]
for the colimit cocones. First we contemplate the diagram
\[ \xygraph{!{0;(3,0):(0,.333)::} {\colim_{m{<}n} T^2Q_m}="tl" [r] {\colim_{m{<}n} TQ_m}="tm" [r] {\colim_{m{<}n} Q_m}="tr" [d] {Q_n}="br" [l] {TQ_n}="bm" [l] {T^2Q_n}="bl" [d] {T^2Q_{n{+}1}}="bbl" [r] {TQ_{n{+}1}}="bbm" [r] {Q_{n{+}1}}="bbr" [d] {Q_{n{+}2}}="bbbr"
"tl":@<1ex>"tm"^-{\mu_{<{n}}}:@<1ex>"tr"^-{v_{<{n}}} "tl":@<-1ex>"tm"_-{(Tv)_{<{n}}}:@<-1ex>@{<-}"tr"_-{\eta_{<{n}}} "bl":"bm"_-{\mu}:@{<-}"br"_-{\eta} "bbl":"bbm"_-{\mu}:@{<-}"bbr"_-{\eta} "tl":"bl"_{o_{n,2}} "tm":"bm"^{o_{n,1}} "tr":@{=}"br" "bl":"bbl"_{T^2q_n} "bm":"bbm"^{Tq_n} "br":"bbr"^{q_n}:"bbbr"^{q_{n{+}1}} "bl":"bbm"_-{Tv_n}:"bbbr"_-{v_{n{+}1}} "bm":"bbr"_-{v_n}} \]
and in general one has
\begin{equation}\label{eq:succ-limit} T(q_n)o_{n,1}(Tv)_{{<}n} = T(v_n)o_{n,2}. \end{equation}
To prove this note that from the definitions of $q_m$ and $q_n$ and the naturality of the $q_{m,n}$ in $m$, one may show easily that $v_nT(q_{m,n})=q_nq_{m{+}1,n}v_m$, and from this last equation and all the definitions it is easy to show that
\[ T(q_n)o_{n,1}(Tv)_{{<}n}q''_{m,n} = T(v_n)o_{n,2}q''_{m,n} \]
for all $m < n$ from which (\ref{eq:succ-limit}) follows.

Suppose that $o_{n,1}$ and $o_{n,2}$ are isomorphisms. Then define $q'_{n}:Q_{n{+}1} \to Q_n$ as the unique map such that $q'_nv_no_{n,1}=q_nv_{{<}n}$. It follows easily that $q'_n=q^{-1}_n$. From (\ref{eq:succ-limit}) and the invertibility of $o_{n,2}$ it follows easily that $v_n\mu=v_nT(q_n^{-1})T(v_n)$ and so there is a unique $q'_{n{+}1}$ such that $q'_{n{+}1}v_{n{+}2}=v_nT(q_n^{-1})$, from which it follows easily that $q'_{n{+}1}=q_{n{+}1}^{-1}$.
\end{proof}
\begin{lemma}\label{really-stable}
If the sequence stabilises at $n$ then it stabilises at any $m \geq n$, and moreover one has an isomorphism of sequences between the given sequence $(Q_m,q_m)$ and the following one:
\[ \xygraph{{Q_0}="p1" [r] {...}="p2" [r] {Q_n}="p3" [r] {Q_n}="p4" [r] {...}="p5" "p1":"p2"^-{q_0}:"p3":"p4"^-{\id}:"p5"^-{\id}} \]
\end{lemma}
\begin{proof}
We show for $m \geq n$ that $q_m$ and $q_{m{+}1}$ are isomorphisms, and provide the component isomorphisms $i_m:Q_m \to Q_n$ of the required isomorphism of sequences, by transfinite induction on $m$. We define $i_m$ to be the identity when $m \leq n$. In the initial step $m=n$, $q_m$ and $q_{m{+}1}$ are isomorphisms by hypothesis and we define $i_{n{+}1}=q_n$. In the inductive step when $m \geq n$ is a non-limit ordinal, we must show that $q_{m{+}2}$ is an isomorphism and define $i_{m{+}2}=q_{m{+}1}i_{m{+}1}$. The key point is that
\begin{equation}\label{eq:key} v_{m{+}1}\mu = v_{m{+}1}T(q^{-1}_{m{+}1})T(v_{m{+}1}) \end{equation}
because with this equation in hand one defines $q'_{m{+}2}:Q_{m{+}3} \to Q_{m{+}2}$ as the unique morphism satisfying $q'_{m{+}2}v_{m{+}2}T(q_{m{+}1})=v_{m{+}1}$ using the universal property of $v_{m{+}2}$, and then it is routine to verify that $q'_{m{+}2}=q^{-1}_{m{+}2}$. So for the inductive step it remains to verify (\ref{eq:key}). But we have
\[ \begin{array}{rll} v_{m{+}1}{\mu}T^2(q_m) &=& v_{m{+}1}T(q_m)\mu = v_{m{+}1}T(v_m) = v_{m{+}1}T(q^{-1}_{m{+}1})T(q_{m{+}1}v_m) \\
&=& v_{m{+}1}T(q^{-1}_{m{+}1})T(v_{m{+}1})T^2(q_m) \end{array} \]
and so (\ref{eq:key}) follows since $q_m$ is an isomorphism. In the case where $m$ is a limit ordinal, we have stabilisation at $m'$ established whenever $n \leq m' < m$ by the induction hypothesis. Thus the colimit defining $Q_m$ is absolute (ie preserved by all functors) since its defining sequence from the position $n$ onwards consists only of isomorphisms. Thus $q_m$ and $q_{m{+}1}$ are isomorphisms by lemma(\ref{stable-limit}). By induction, the previously constructed $i_{m'}$'s provide a cocone on the defining diagram of $Q_m$ with vertex $Q_n$, thus one induces the isomorphism $i_m$ compatible with the earlier $i_{m'}$'s and defines $i_{m{+}1}=q_mi_m$.
\end{proof}
\begin{lemma}\label{coeq-when-stable}
If the sequence stabilises at $n$ then $(Q_n,q_n^{-1}v_n)$ is a $T$-algebra and
\[ q_{<{n}}:(B,b) \to (Q_n,q_n^{-1}v_n) \]
is the coequaliser of $f$ and $g$ in $V^T$.
\end{lemma}
\begin{proof}
The unit law for $(Q_n,q_n^{-1}v_n)$ is immediate from the definition of $q_n$ and the associative law is the commutativity of the outside of the diagram on the left
\[ \xygraph{{\xybox{\xygraph{!{0;(1.5,0):(0,.7)::} {T^2Q_n}="tl" [r(2)] {TQ_n}="tr" [d] {Q_{n{+}1}}="mr" [d] {Q_n}="br" [l] {Q_{n{+}1}}="bm" [l] {TQ_n}="bl" [u] {TQ_{n{+}1}}="ml" "bm" [u(.75)] {Q_{n{+}2}}="mm" [u(.75)] {TQ_{n{+}1}}="tm"
"tl":"tr"^-{\mu}:"mr"^-{v_n}:"br"^-{q_n^{-1}}:@{<-}"bm"^-{q_n^{-1}}:@{<-}"bl"^-{v_n}:@{<-}"ml"^-{Tq_n^{-1}}:@{<-}"tl"^-{Tv_n}
"tr":"tm"^-{Tq_n}:"mm"_-{v_{n{+}1}}:"bm"_-{q_{n{+}1}^{-1}} "ml":"mm"_-{v_{n{+}1}}:"mr"^-{q_{n{+}1}^{-1}}}}}
[r(5)u(.1)]
{\xybox{\xygraph{{TB}="tl" [r(2)] {TQ_n}="tr" [d] {Q_{n{+}1}}="m" [d] {Q_n}="br" [l(2)] {B}="bl" "tl":"tr"^-{Tq_{{<}n}}:"m"^-{v_n}:"br"^-{q_n^{-1}}:@{<-}"bl"^-{q_{{<}n}}(:@{<-}"tl"^-{b},:"m"^-{q_{{<}n{+}1}}) "m":@{}"tl"|{(I)})}}}} \]
the regions of which evidently commute. The commutativity of the outside diagram on the right exhibits $q_{{<}n}$ as a $T$-algebra map, and this follows immediately from the commutativity of the region labelled (I).

The equational form of (I) says $q_{{<}n{+}1}b=v_nT(q_{{<}n})$ and we now proceed to prove this by transfinite induction on $n$. The case $n=0$ is just the statement $v_0=q_0b$. The inductive step comes out of the calculation
\[ q_{{<}n{+}2}b = q_{n{+}1}q_{{<}n{+}1}b = q_{n{+}1}v_nT(q_{{<}n}) = v_{n{+}1}T(q_nq_{{<}n}) = v_{n{+}1}T(q_{{<}n{+}1}) \]
and since $Tq_{{<}n}=o_{n,1}q'_{0,n}$. The case where $n$ is a limit ordinal is the commutativity of the outside of
\[ \xygraph{!{0;(1.4,0):} {TB}="tb" ([ur] {B}="b" [r(3)] {Q_{n{+}1}}="qnp1",[dr] {TB}="tb2" [r(3)] {\col \, TQ_m}="ctq" [ur] {TQ_n}="tqn")
"tb"(:"b"^-{b}:"qnp1"^-{q_{{<}n{+}1}},:"tb2"_-{\id}|-{}="mla":@/_{1pc}/"ctq"_-{q'_{0,n}}|-{}="mma":"tqn"_-{o_{n,1}}|-{}="mra":"qnp1"_-{v_n} "b":@{}"ctq"|*{TB}="m")
"m" (:@{}"b"|(.6)*{Q_n}="ptl",:@{}"qnp1"|(.6)*{TQ_n}="ptr",:@{}"mra"|(.6)*{\col \, TQ_m}="pmr",:@{}"mma"|(.6)*{\col \, T^2Q_m}="pb",:@{}"mla"|(.6)*{T^2B}="pml")
"b":"ptl"_(.6){q_{{<}n}}:"ptr"^-{\eta}:"qnp1"^(.4){v_n}
"tb":"pml"^-{\eta}:"m"^-{Tb}(:"ptr"^(.4){Tq_{{<}n}},:"pmr"^-{q'_{0,n}})
"pml":"pb"^-{q''_{0,n}}:"pmr"^-{(Tv)_{{<}n}}:"ptr"^-{o_{n,1}}
"pml":"tb2"^-{\mu} "pb":"ctq"^-{\mu_{{<}n}}} \]
the regions of which evidently commute. Thus $q_{{<}n}$ is indeed a $T$-algebra map.

To see that $q_{{<}n}$ is a coequaliser let $h:(B,b) \to (C,c)$ such that $hf=hg$. For each ordinal $m$ we construct $h_m : Q_m \to C$ such that $h_{m{+}1}v_m=cT(h_m)$ for all $m$ by transfinite induction on $m$. When $m=0$ we define $h_0=h$ and $h_1$ as unique such that $h_1q_0=h$. The equation $h_1v_0=cT(h)$ is easily verified. For the inductive step we note that the commutativity of
\[ \xygraph{{T^2Q_n}="ml" [ur] {TQ_n}="tl" [r(2)] {TC}="tr" [dr] {C}="mr" [dl] {TC}="br" [l(2)] {TQ_{n{+}1}}="bl" [ur] {T^2C}="m"
"ml":"tl"^-{\mu}:"tr"^-{Th_n}:"mr"^-{c}:@{<-}"br"^-{c}:@{<-}"bl"^-{Th_{n{+}1}}:@{<-}"ml"^-{Tv_n} "m"(:@{<-}"ml"_-{T^2h_n},:"tr"^-{\mu},:"br"^-{Tc})} \]
and the universal property of $v_{m{+}1}$ ensures there is a unique $h_{m{+}2}$ such that $h_{m{+}2}v_{m{+}1}=cT(h_{m{+}1})$. When $m$ is a limit ordinal it follows from all the definitions that
\[ cT(h_m)o_{m,1}\mu_{{<}m}q''_{m',m} = cT(h_m)o_{m,1}(Tv)_{{<}m}q''_{m',m} \]
for all $m' < m$, and so $cT(h_m)o_{m,1}\mu_{{<}m} = cT(h_m)o_{m,1}(Tv)_{{<}m}$ and so by the universal property of $v_{m{+}1}$ there is a unique $h_{m{+}1}$ such that $h_{m{+}1}v_m=cT(h_m)$. The sequence of $h_m$'s just constructed is clearly unique such that $h_0=h$ and $h_{m{+}1}v_m=cT(h_m)$. It follows immediately that $h_n$ is a $T$-algebra map, and that $h_nq_{{<}n}=h$. Conversely given $h':Q_n \to C$ such that $h'q_{{<}n}=h$, one constructs $h'_m:Q_m \to C$ as $h'_{m}=h'q_{m,n}$, and it follows easily that $h'_0=h$, $h'_{m{+}1}v_m=cT(h'_m)$ and $h'_n=h'$ whence $h'_m=h_m$ and so $h=h'$.
\end{proof}
\noindent From these results we recover the usual theorem on the construction of coequalisers of algebras of accessible monads.
\begin{theorem}\label{thm:coeq-Talg}
Let $V$ be a category with filtered colimits and coequalisers, $T$ be a monad on $V$ and
\[ \xygraph{!{0;(2,0):} {(A,a)}="l" [r] {(B,b)}="r" "l":@<-1ex>"r"_-{g} "l":@<1ex>"r"^-{f}} \]
be morphisms in $V^T$. If $T$ is $\lambda$-accessible for some regular cardinal $\lambda$, then $q_{<{n}}$ as constructed above is the coequaliser of $f$ and $g$ in $V^T$ for any ordinal $n$ such that $|n| \geq \lambda$.
\end{theorem}
\begin{proof}
Take the smallest such ordinal $n$ -- it is necessarily a limit ordinal, and $T$ and $T^2$ by hypothesis preserve the defining colimit of $Q_n$. Thus by lemmas(\ref{stable-limit}) and (\ref{coeq-when-stable}) the result follows in this case, and in general by lemmas(\ref{really-stable}) and (\ref{coeq-when-stable}).
\end{proof}
Finally we mention the well-known special case when the above transfinite construction is particularly simple, that will be worth remembering.
\begin{proposition}\label{prop:simple-coeq}
Let $V$ be a category with filtered colimits and coequalisers, $T$ be a monad on $V$ and
\[ \xygraph{!{0;(2,0):} {(A,a)}="l" [r] {(B,b)}="r" "l":@<-1ex>"r"_-{g} "l":@<1ex>"r"^-{f}} \]
be morphisms in $V^T$. If $T$ and $T^2$ preserve the coequaliser of $f$ and $g$ in $V$, then the sequence $(Q_n,q_n)$ stabilises at $1$. Denoting by $w:TQ_1 \to Q_1$ the unique map such that $wT(q_0)=q_0b$, $q_0:(B,b) \to (Q_1,w)$ is the coequaliser of $f$ and $g$ in $V^T$.
\end{proposition}
\begin{proof}
Refer to the diagram in $V$ above that describes the first few steps of the construction of $(Q_n,q_n)$. Since $q_0$ and $T^2q_0$ are epimorphisms, the $T$-algebra axioms for $(Q_1,w)$ follow from those for $(B,b)$, and $q_0$ is a $T$-algebra map by definition. Thus $w$ is the coequaliser in $V$ of $\mu_{Q_1}$ and $Tw$, and since $T^2q_0$ is an epimorphism it is also the coequaliser of $\mu_{Q_1}T^2(q_0)$ and $T(w)T^2(q_0)=Tv_0$, but so is $v_1$, and so $q_1$ is the canonical isomorphism between them. To see that $q_2$ is also invertible, apply the same argument with the composite $q_1q_0$ in place of $q_0$. The result now follows by lemma(\ref{coeq-when-stable}).
\end{proof}
%

\subsection{Monads induced by monad morphisms}\label{ssec:Dubuc}
Suppose that $(M,\eta^M,\mu^M)$ and $(S,\eta^S,\mu^S)$ are monads on a category $V$, and $\phi:M{\rightarrow}S$ is a morphism of monads. Then one has an obvious forgetful functor $\phi^*:V^S{\rightarrow}V^M$ and one can ask whether $\phi^*$ has a left adjoint which, when it exists, we denote as $\phi_!$. By the Dubuc adjoint triangle theorem \cite{Dubuc}, one may compute the value of $\phi_!$ at an $M$-algebra $(X,x:MX{\rightarrow}X)$ as a reflexive coequaliser
\[ \xygraph{!{0;(3,0):}
{(SMX,\mu^S_{MX})}="l" [r] {(SX,\mu^S_X)}="m" [r] {\phi_!(X,x)}="r"
"l":@<2ex>"m"^-{\mu^S_XS(\phi_X)}:"l"|-{S\eta^M_X}:@<-2ex>"m"_-{Sx}:"r"^-{q_{(X,x)}}} \]
in $V^S$, when this coequaliser exists. Thus by theorem(\ref{thm:coeq-Talg}) for the existence of $\phi_!$ it suffices that $V$ admit filtered colimits and coequalisers, and $S$ be accessible. With the aid of section(\ref{ssec:coequalisers}) we shall now give an explicit description of the composite $U^S\phi_!$ under these hypotheses. To do so we shall construct morphisms
\[ \begin{array}{lccr} {v_{n,X,x} : SQ_n(X,x) \rightarrow Q_{n{+}1}(X,x)} &&& {q_{n,X,x}:Q_n(X,x) \to Q_{n{+}1}(X,x)} \end{array} \]
\[ \begin{array}{c} {q_{{<}n,X,x}:SX \to Q_n(X,x)}  \end{array} \]
starting with $Q_0(X,x)=SX$ by transfinite induction on $n$.
\\ \\
{\bf Initial step}. Define $q_{{<}0}$ to be the identity, $q_0$ to be the coequaliser of $\mu^S(S\phi)$ and $Sx$, $q_{<{1}}=q_0$ and $v_0=q_0b$. Note also that $q_0=v_0\eta^S$.
\\ \\
{\bf Inductive step}. Assuming that $v_n$, $q_n$ and $q_{<{n{+}1}}$ are given, we define $v_{n{+}1}$ to be the coequaliser of $S(q_n)(\mu^SQ_n)$ and $Sv_n$, $q_{n{+}1}=v_{n{+}1}(\eta^SQ_{n{+}1})$ and $q_{<{n{+}2}}=q_{n{+}1}q_{<{n{+}1}}$.
\\ \\
{\bf Limit step}. Define $Q_n(X,x)$ as the colimit of the sequence given by the objects $Q_m(X,x)$ and morphisms $q_m$ for $m < n$, and $q_{<{n}}$ for the component of the universal cocone at $m=0$.
\[ \xygraph{!{0;(3,0):(0,.333)::} {\colim_{m{<}n} S^2Q_m}="tl" [r] {\colim_{m{<}n} SQ_m}="tm" [r] {\colim_{m{<}n} Q_m}="tr" [d] {Q_n}="br" [l] {SQ_n}="bm" [l] {S^2Q_n}="bl" "tl":@<1ex>"tm"^-{\mu_{<{n}}}:@<1ex>"tr"^-{v_{<{n}}} "tl":@<-1ex>"tm"_-{(Sv)_{<{n}}}:@<-1ex>@{<-}"tr"_-{\eta_{<{n}}} "bl":"bm"_-{\mu}:@{<-}"br"_-{\eta} "tl":"bl"_{o_{n,2}} "tm":"bm"^{o_{n,1}} "tr":@{=}"br"} \]
We write $o_{n,1}$ and $o_{n,2}$ for the obstruction maps measuring the extent to which $S$ and $S^2$ preserve the colimit defining $Q_n(X,x)$. We write $\mu^S_{<{n}}$, $(Sv)_{<{n}}$, $v_{<{n}}$ and $\eta^S_{<{n}}$ for the maps induced by the $\mu^SQ_m$, $Sv_m$, $v_m$ and $\eta^SQ_m$ for $m < n$ respectively. Define $v_n$ as the coequaliser of $o_{n,1}\mu_{<{n}}$ and $o_{n,1}(Sv)_{<{n}}$, $q_n=v_n({\eta^S}Q_n)$ and $q_{<{n{+}1}}=q_nq_{<{n}}$.
\\ \\
Instantiating theorem(\ref{thm:coeq-Talg}) to the present situation gives
\begin{corollary}\label{cor:explicit-phi-shreik}
Suppose that $V$ admits filtered colimits and coequalisers, $M$ and $S$ are monads on $V$, $\phi:M{\rightarrow}S$ is a morphism of monads, and $(X,x)$ is an $M$-algebra. If moreover $S$ is $\lambda$-accessible for some regular cardinal $\lambda$, then for any ordinal $n$ such that $|n| \geq \lambda$
one may take
\[ \begin{array}{lccr} {\phi_!(X,x)=(Q_n(X,x),q_n^{-1}v_n)} &&& {q_{<n} : (SX,\mu_X) \to (Q_n(X,x),q_n^{-1}v_n)} \end{array} \]
as an explicit definition of $\phi_!(X,x)$ and the associated coequalising map in $V^S$ coming from the Dubuc adjoint triangle theorem.
\end{corollary}
\noindent and instantiating proposition(\ref{prop:simple-coeq}) to the present situation gives
\begin{corollary}\label{cor:phi-shreik-simple}
Suppose that under the hypotheses of corollary(\ref{cor:explicit-phi-shreik}) that $S$ and $S^2$ preserve the coequaliser of $\mu_{X}^{S}S(\phi_X)$ and $Sx$ in $V$. Then the sequence $(Q_n,q_n)$ stabilises at $1$, and writing $w:SQ_1 \to Q_1$ for the unique map such that $wS(q_0)=q_0\mu_{X}^{S}$, one may take
\[ \begin{array}{lccr} {\phi_!(X,x)=(Q_1(X,x),w)} &&& {q_0 : (SX,\mu_X) \to (Q_1(X,x),w)} \end{array} \]
as an explicit definition of $\phi_!(X,x)$ and the associated coequalising map in $V^S$.
\end{corollary}
\begin{remark}
Here is a degenerate situation in which corollary(\ref{cor:phi-shreik-simple}) applies. Since $U^M\phi^*=U^S$ we have $\phi_!F^M \iso F^U$, but another way to view this isomorphism as arising is to apply the corollary in the case where $(X,x)$ is a free $M$-algebra, say $(X,x)=(MZ,\mu^M_Z)$, for in this case one has the dotted arrows in
\[ \xygraph{!{0;(2,0):} {SM^2Z}="l" [r] {SMZ}="m" [r] {SZ}="r" "l":@<1.5ex>"m"^-{(\mu^SM)(S{\phi}M)} "l":@<-1.5ex>"m"^-{S\mu^M} "l":@{<.}@<-3.5ex>"m"_-{SM\eta^M} "m":"r"^-{\mu^SS(\phi)} "m":@{<.}@<-2ex>"r"_-{S\eta^M}} \]
exhibiting $\mu_Z^SS(\phi_Z)$ as a split coequaliser, and thus absolute.
\end{remark}
Let us denote by $(T,\eta^T,\mu^T)$ the monad on $V^M$ induced by the adjunction $\phi_! \ladj \phi^*$. We now give an explicit description of this monad. Let $(X,x)$ be in $V^M$, suppose $S$ is $\lambda$-accessible and fix an ordinal $n$ such that $|n| \geq \lambda$. Then by corollary(\ref{cor:explicit-phi-shreik}) one may take
\[ \begin{array}{lccr} {T(X,x) = (Q_n(X,x),a(X,x)\phi_{Q_n(X,x)})} &&& {a(X,x) = (q^{-1}_n)_{Q_n(X,x)} (v_{n})_{Q_n(X,x)}} \end{array} \]
as the definition of the endofunctor $T$. Note that $(Q_n(X,x),a(X,x))$ is just a more refined notation for $\phi_!(X,x)$. Referring to the diagram
\[ \xygraph{!{0;(2,0):(0,.6)::} {M^2X}="tl" [r] {MX}="tm" [r] {X}="tr" [d] {Q_n}="br" [l] {SX}="bm" [l] {SMX}="bl" "tl":@<-1ex>"tm"_-{Mx} "tl":@<1ex>"tm"^-{\mu^M_X}:"tr"^-{x} "bl":@<-1ex>"bm"_-{Sx} "bl":@<1ex>"bm"^-{\mu^S_XS(\phi_X)}:"br"_-{q_{{<}n}} "tl":"bl"_{\phi_{MX}} "tm":"bm"^{\phi_X} "tr":@{.>}"br"^{\eta^T_{(X,x)}}} \]
one may define the underlying map in $V$ of $\eta^T_{(X,x)}$ as the unique map making the square on the right commute. This makes sense since the top row is a coequaliser in $V$. Via the evident $M$-algebra structures on each of the objects in this diagram, one may in fact interpret the whole diagram in $V^M$ with the top row now being the canonical presentation coequaliser for $(X,x)$, and this is why $\eta^T_{(X,x)}$ is an $M$-algebra map. The proof that $\eta^T_{(X,x)}$ possesses the universal property of the unit of $\phi_! \ladj \phi^*$ is straight forward and left to the reader. As for $\mu^T_{(X,x)}$, it is induced from the following situation in $V^S$:
\[ \xygraph{{(SMQ_n,\mu^S)}="l" [r(2.5)] {(SQ_n,\mu^S)}="m" [r(4)] {(Q_n(Q_n,a\phi),a(Q_n,a\phi))}="r" [dl] {(Q_n(X,x),a(X,x))}="b" "l":@<-1ex>"m"_-{\mu^SS(\phi)} "l":@<1ex>"m"^-{S(a(X,x)\phi)}(:"r"^-{(q_{{<}n})_{(Q_n,a\phi)}}:@{.>}"b"^(.35){\mu^T_{(X,x)}},:"b"_{a(X,x)})} \]
Since by definition $\mu^T_{(X,x)}$ underlies an $S$-algebra map, to finish the proof that our definition really does describe the multiplication of $T$, it suffices by the universal property of $\eta^T$ to show that $\mu^T_{(X,x)}\eta^T_{T(X,x)}$ is the identity, and this is easily achieved using the defining diagrams of $\mu^T$ and $\eta^T$ together.

Let us now describe $\eta^T$ and (especially) $\mu^T$ in terms of the transfinite data that gives $Q_n(X,x)$. To do so we shall for each ordinal $m$ provide
\[ \begin{array}{lccr} {\eta^{(m{+}1)}_{(X,x)} : X \to Q_{m{+}1}(X,x)} &&& {\mu^{(m)}_{(X,x)} : Q_m(Q_n(X,x),a(X,x)\phi) \to Q_n(X,x)} \end{array} \]
and $\mu^{(m{+}1)}_{(X,x)}$ in $V$ such that $\mu^{(m{+}1)}v_m=a(X,x)S(\mu^{(m)})$, by transfinite induction on $m$.
\\ \\
{\bf Initial step}. Define $\mu^{(0)}_{(X,x)}$ to be the identity, and $\eta^{(1)}_{(X,x)}$ and $\mu^{(1)}_{(X,x)}$ as the unique morphisms such that
\[ \begin{array}{lccr} {\eta^{(1)}_{(X,x)}x=(q_0)_{(X,x)}\phi_X} &&& {\mu^{(1)}_{(X,x)}(q_0)_{(Q_n,a\phi)}=a(X,x)} \end{array} \]
by the universal properties of $x$ and $q_0$ (as the evident coequalisers) respectively.
\\ \\
{\bf Inductive step}. Define $\eta^{(m{+}2)}=q_{m{+}1}\eta^{(m{+}1)}$ and $\mu^{(m{+}2)}$ as the unique map satisfying $\mu^{(m{+}2)}v_{m{+}1}=a(X,x)S(\mu^{(m{+}1)})$ using the universal property of $v_{m{+}1}$ as a coequaliser.
\\ \\
{\bf Limit step}. When $m$ is a limit ordinal define $\eta^{(m)}_{(X,x)}$ and $\mu^{(m)}_{(X,x)}$ as the maps induced by the $\eta^{(r)}$ and $\mu^{(r)}$ for $r<m$ and the universal property of $Q_m(X,x)$ as the colimit of the sequence of the $Q_r$ for $r<m$. Then define $\eta^{(m{+}1)}=q_m\eta^{(m)}$ and $\mu^{(m{+}2)}$ as the unique map satisfying $\mu^{(m{+}2)}v_{m{+}1}=a(X,x)S(\mu^{(m{+}1)})$ using the universal property of $v_{m{+}1}$ as a coequaliser.
\\ \\
The fact that the induction just given was obtained by unpacking the descriptions of $\eta^T$ and $\mu^T$ of the previous paragraph in terms of the transfinite construction of the endofunctor $T$ (ie the $Q_m(X,x)$), is expressed by
\begin{corollary}\label{cor:induced-monad-very-explicit}
Suppose that $V$ admits filtered colimits and coequalisers, $M$ and $S$ are monads on $V$, $\phi:M{\rightarrow}S$ is a morphism of monads, and $(X,x)$ is an $M$-algebra. If moreover $S$ is $\lambda$-accessible for some regular cardinal $\lambda$, then for any ordinal $n$ such that $|n| \geq \lambda$
one may take
\[ \begin{array}{lcccr} {T(X,x) = (Q_n(X,x),a(X,x)\phi_{Q_n(X,x)})} && {\eta^T_{(X,x)}=\eta^{(n)}_{(X,x)}} && {\mu^T_{(X,x)}=\mu^{(n)}_{(X,x)}} \end{array} \]
as constructed above as an explicit description underlying endofunctor, unit and multiplication of the monad generated by the adjunction $\phi_! \ladj \phi^*$.
\end{corollary}
\noindent and the simplification coming from proposition(\ref{prop:simple-coeq}) gives
\begin{corollary}\label{cor:vexp-simple}
Under the hypotheses of corollary(\ref{cor:induced-monad-very-explicit}), if for $(X,x) \in V^M$, $S$ and $S^2$ preserve the coequaliser of $\mu_{X}^{S}S(\phi_X)$ and $Sx$ in $V$, then one may take
\[ \begin{array}{lcccr} {T(X,x) = (Q_1(X,x),w\phi)} && {\eta^T_{(X,x)}=\eta^{(1)}_{(X,x)}} && {\mu^T_{(X,x)}=\mu^{(1)}_{(X,x)}} \end{array} \]
with $w$ as constructed in corollary(\ref{cor:phi-shreik-simple}).
\end{corollary}
%

\section*{Acknowledgements}\label{sec:Acknowledgements}
We would like to acknowledge
Clemens Berger, Richard Garner, Andr\'{e} Joyal, Steve Lack, Joachim Kock, Jean-Louis Loday, Paul-Andr\'{e} Melli\`{e}s, Ross Street and Dima Tamarkin
for interesting discussions on the substance of this paper. We would also like to acknowledge the Centre de Recerca Matem\`{a}tica in Barcelona for the generous hospitality and stimulating environment provided during the thematic year 2007-2008 on Homotopy Structures in Geometry and Algebra.
The first author would like to acknowledge the financial support on different stages of this project of the Scott Russell Johnson Memorial Foundation, the Australian Research Council (grant No.DP0558372) and L'Universit\'{e} Paris 13.
The second author would like to acknowledge the support of the ANR grant no. ANR-07-BLAN-0142.
The third author would like to acknowledge the laboratory PPS (Preuves Programmes Syst\`{e}mes) in Paris, the Max Planck Institute in Bonn, the IHES and the Macquarie University Mathematics Department for the excellent working conditions he enjoyed during this project.

\end{document}